\documentclass[12pt,english]{article}
\usepackage[T1]{fontenc}
\usepackage[latin9]{inputenc}
\usepackage{geometry}
\usepackage{mathrsfs}
\usepackage{amscd}
\usepackage{amsmath}
\usepackage{amsthm}
\usepackage{amssymb}

\makeatletter

\providecommand{\tabularnewline}{\\}

\numberwithin{equation}{section}
\numberwithin{figure}{section}
\theoremstyle{plain}
\newtheorem{thm}{\protect\theoremname}[section]
\theoremstyle{plain}
\newtheorem{lem}[thm]{\protect\lemmaname}
\theoremstyle{plain}
\newtheorem{prop}[thm]{\protect\propositionname}
\theoremstyle{plain}
\newtheorem{cor}[thm]{\protect\corollaryname}
\theoremstyle{plain}
\newtheorem{assumption}[thm]{\protect\assumptionname}
\theoremstyle{definition}
\newtheorem*{example*}{\protect\examplename}
\theoremstyle{definition}
\newtheorem{defn}[thm]{\protect\definitionname}
\theoremstyle{remark}
\newtheorem{rem}[thm]{\protect\remarkname}
\theoremstyle{plain}
\newtheorem{conjecture}[thm]{\protect\conjecturename}

\date{}

\makeatother

\usepackage{babel}
\providecommand{\assumptionname}{Assumption}
\providecommand{\conjecturename}{Conjecture}
\providecommand{\corollaryname}{Corollary}
\providecommand{\definitionname}{Definition}
\providecommand{\examplename}{Example}
\providecommand{\lemmaname}{Lemma}
\providecommand{\propositionname}{Proposition}
\providecommand{\remarkname}{Remark}
\providecommand{\theoremname}{Theorem}

\begin{document}
\title{Weyl group action on Radon hypergeometric function and its symmetry }
\author{Hironobu Kimura\\
 Department of Mathematics, Graduate School of Science and\\
 Technology, Kumamoto University}

\maketitle
\global\long\def\R{\mathbb{R}}%
 
\global\long\def\al{\alpha}%
\global\long\def\be{\beta}%
 
\global\long\def\ga{\gamma}%
 
\global\long\def\de{\delta}%
 
\global\long\def\expo{\mathrm{exp}}%
 
\global\long\def\f{\varphi}%
 
\global\long\def\W{\Omega}%
 
\global\long\def\lm{\lambda}%
\global\long\def\te{\theta}%
 
\global\long\def\C{\mathbb{C}}%
 
\global\long\def\Z{\mathbb{Z}}%
 
\global\long\def\Ps{\mathbb{P}}%
 
\global\long\def\De{\Delta}%
 
\global\long\def\cbatu{\mathbb{C}^{\times}}%
 
\global\long\def\La{\Lambda}%
 
\global\long\def\vt{\vartheta}%
 
\global\long\def\G{\Gamma}%
\global\long\def\GL#1{\mathrm{GL}(#1)}%
 
\global\long\def\gras{\mathrm{Gr}}%
 
\global\long\def\diag{\mathrm{diag}}%
 
\global\long\def\tr{\,\mathrm{^{t}}}%
 
\global\long\def\re{\mathrm{Re}}%
 
\global\long\def\im{\mathrm{Im}}%
 
\global\long\def\norm{\mathscr{N}(n)}%
 
\global\long\def\sp{\mathrm{sp}}%
 
\global\long\def\rank{\mathrm{rank}}%
 
\global\long\def\tH{\tilde{H}}%
 
\global\long\def\mat{\mathrm{Mat}}%
 
\global\long\def\lto{\longrightarrow}%
 
\global\long\def\S{\mathfrak{S}}%
 
\global\long\def\cO{\mathcal{O}}%
 
\global\long\def\gl{\mathfrak{gl}}%
\global\long\def\gee{\mathfrak{g}}%
 
\global\long\def\Tr{\,\mathrm{Tr}}%
 
\global\long\def\cS{\mathcal{S}}%
 
\global\long\def\ad{\mathrm{ad}}%
 
\global\long\def\Ad{\mathrm{Ad}}%
 
\global\long\def\ha{\mathfrak{h}}%
 
\global\long\def\fj{\mathfrak{j}}%
 
\global\long\def\mrn{\mat'(r,N)}%
 
\global\long\def\sgn{\mathrm{sgn}}%
 
\global\long\def\hlam{H_{\lambda}}%
 
\global\long\def\ghyp{\,_{2}F_{1}}%
 
\global\long\def\mnm{\mat'(m,N)}%
\global\long\def\cP{\mathcal{P}}%
 
\global\long\def\auto{\mathrm{Aut}}%
 
\global\long\def\la{\langle}%
 
\global\long\def\ra{\rangle}%
 
\global\long\def\Ai{\mathrm{Ai}}%
 
\global\long\def\adj{\mathrm{Ad}}%
 
\global\long\def\yn{\mathbf{Y}_{n}}%
 
\global\long\def\eq{\mathcal{I}}%
 
\global\long\def\hyp#1#2{\, _{#1}F_{#2}}%
\global\long\def\jro{J_{r}^{\circ}}%
 
\global\long\def\jroo{\mathfrak{j}_{r}^{\circ}}%
\global\long\def\jor#1{J^{\circ}(#1)}%
 
\global\long\def\pa{\partial}%
 
\global\long\def\bx{\mathbf{x}}%
 
\global\long\def\fa{\mathfrak{a}}%
 
\global\long\def\bb{\mathbf{b}}%
 
\global\long\def\bz{\mathbf{z}}%
 
\global\long\def\etr{\mathrm{etr}}%
 
\global\long\def\cR{\mathcal{R}}%
 
\global\long\def\cL{\mathcal{L}}%
 
\global\long\def\wm{\omega}%
 
\global\long\def\lm{\lambda}%
\global\long\def\te{\theta}%
 
\global\long\def\ep{\epsilon}%
 
\global\long\def\fl{\mathrm{Flag}}%
 
\global\long\def\vep{\varepsilon}%
 
\global\long\def\Gl{\mathrm{GL}}%
 
\global\long\def\sm{\sigma}%
 
\global\long\def\ini{\mathrm{in}_{\prec}}%
 
\global\long\def\Span{\mathrm{span}}%
 
\global\long\def\cS{\mathcal{S}}%
 
\global\long\def\cL{\mathcal{L}}%
 
\global\long\def\cF{\mathcal{F}}%
 
\global\long\def\rank{\mathrm{rank}}%
 
\global\long\def\tH{\tilde{H}}%
 
\global\long\def\mat{\mathrm{Mat}}%
 
\global\long\def\lto{\longrightarrow}%
 
\global\long\def\Si{\mathfrak{S}}%
 
\global\long\def\cO{\mathcal{O}}%
 
\global\long\def\gl{\mathfrak{gl}}%
\global\long\def\gee{\mathfrak{g}}%
 
\global\long\def\Tr{\,\mathrm{Tr}}%
 
\global\long\def\ad{\mathrm{ad}}%
 
\global\long\def\Ad{\mathrm{Ad}}%
 
\global\long\def\ha{\mathfrak{h}}%
 
\global\long\def\fj{\mathfrak{j}}%
 
\global\long\def\mrn{\mat'(r,N)}%
 
\global\long\def\sgn{\mathrm{sgn}}%
 
\global\long\def\hlam{H_{\lambda}}%
 
\global\long\def\mnm{\mat'(m,N)}%
\global\long\def\cP{\mathcal{P}}%
 
\global\long\def\herm{\mathscr{H}(r)}%
 
\begin{abstract}
For positive integers $r,n,N:=rn$, we consider the Radon hypergeometric
function (Radon HGF) associated with a partition $\lm$ of $n$ defined
on the Grassmannian $\gras(m,N)$ for $r<m<N$, which is obtained
as the Radon transform of a character of the group $H_{\lm}\subset G:=\GL N$.
We study its symmetry described by the Weyl group analogue $N_{G}(H_{\lm})/H_{\lm}$.
We consider the Hermitian matrix integral analogue of the Gauss HGF
and its confluent family, which are understood as the Radon HGF on
$\gras(2r,4r)$ for partitions $\lm$ of $4$, we apply the result
of symmetry to these particular cases and derive a transformation
formula for the Gauss analogue which is known as a part of ''24 solutions
of Kummer'' for the classical Gauss HGF. We derive a similar transformation
formula for the analogue Kummer's confluent HGF.
\end{abstract}

\section{Introduction}

This paper is a succession of the previous papers \cite{kimura-2}
and \cite{kimura-3} on the Radon hypergeometric function (Radon HGF).
The Radon HGF is an extension of the Gelfand HGF. The Gelfand HGF
was introduced in 1986 by Gelfand \cite{Gelfand} using the Radon
transform. Like as the Gelfand HGF, the Radon HGF is also defined
by the Radon transform. In \cite{kimura-2} we gave the definition
of the Radon HGF of confluent and non-confluent type, and in \cite{kimura-3}
we studied the contiguity relations of the Radon HGF. In this paper,
we discuss the symmetry of the Radon HGF which is described by an
action of a certain analogue of Weyl group as will be explained below.
This symmetry gives a transformation formula when applied to the Hermitian
matrix integral analogue of the Gauss HGF (Section \ref{subsec:exa-gauss}).

Let $r,n$ be positive integers and let $N:=nr$. Let $G=\GL N$ be
the complex general linear group. For any partition $\lm$ of $n$,
we consider the subgroup $H_{\lm}\subset G$ and a character $\chi_{\lm}(\cdot;\al)$
of the universal covering group $\tilde{H}_{\lm}$ which depends on
$\al\in\C^{n}$. When $\lm=(1,\dots,1)$, namely the partition whose
parts are all $1$, $H_{\lm}\simeq(\GL r)^{n}$ and a character is
given by 
\[
\chi_{(1,\dots,1)}(h;\al)=\prod_{1\leq j\leq n}\left(\det h_{j}\right)^{\al_{j}}
\]
with $\al=(\al_{1},\dots,\al_{n})\in\C^{n}$. Then the Radon HGF is,
roughly speaking, defined by 
\[
F_{\lm}(z,\al;C)=\int_{C(z)}\chi_{\lm}(tz;\al)\cdot\tau(t)
\]
as a function on some Zariski open set $Z\subset\mat(m,N)$, where
$m$ is an integer such that $r<m<N$ and $t\in\mat(r,m)$ is the
homogeneous coordinates of the Grassmannian $\gras(r,m)$, the set
of $r$-dimensional subspaces of $\C^{m}$, and $\tau(t)$ is a certain
$r(m-r)$-form in $t$-space. In case $\lm=(1,\dots,1)$, we write
$z=(z^{(1)},\dots,z^{(n)})$ with $z^{(j)}\in\mat(m,r)$, then the
Radon HGF has the form 
\[
F_{(1,\dots,1)}(z,\al;C)=\int_{C(z)}\prod_{1\leq j\leq n}\left(\det tz^{(j)}\right)^{\al_{j}}\cdot\tau(t).
\]
The Weyl group in our context is defined by $W_{\lm}:=N_{G}(H_{\lm})/H_{\lm}$,
where $N_{G}(H_{\lm})$ is the normalizer of $H_{\lm}$ in $G$. One
of our main result is the explicit determination of the Weyl group
$W_{\lm}$ (Theorem \ref{thm:weyl-main-1}). When $\lm=(1,\dots,1)$,
$W_{\lm}$ is isomorphic to the permutation group $\S_{n}$ and the
isomorphism is given by $\S_{n}\ni\sm\mapsto P_{\sm}=(\de_{a,\sm(b)}\cdot1_{r})_{1\leq a,b\leq n}\in G$,
where $1_{r}$ is the identity matrix of size $r$ and $P_{\sm}$
is a permutation matrix in blocks associated with $\sm\in\S_{n}$.
Hence $W_{\lm}$ is a finite group in this case. However, $W_{\lm}$
is not so for $\lm\ne(1,\dots,1)$. It has the form of a semi-direct
product of a continuous group and a finite group.

Let us explain our motivation more concretely explaining the relation
of the Radon HGF to the classical HGFs and to their Hermitian matrix
integral analogues. 

Among the classical HGFs, the Gauss HGF and its confluent family form
an important part. The confluent family consists of Kummer's confluent
HGF, Bessel function, Hermite-Weber function and Airy function. They
are given by the integrals
\begin{align*}
\text{Gauss: } & \int_{C}u^{a-1}(1-u)^{c-a-1}(1-xu)^{-b}du,\\
\text{Kummer: } & \int_{C}e^{xu}u^{a-1}(1-u)^{c-a-1}du,\\
\text{Bessel: } & \int_{C}e^{u-\frac{x}{u}}u^{-c-1}dt=\int_{C'}e^{xu-\frac{1}{u}}u^{c-1}du,\\
\text{Hermite-Weber: } & \int_{C}e^{xu-\frac{1}{2}u^{2}}u^{-a-1}du,\\
\text{Airy: } & \int_{C}e^{xu-u^{3}/3}du,
\end{align*}
with an appropriate path of integration $C$, and each of them is
characterized as a solution of the 2nd order differential equation
on the $1$-dimesional complex projective space $\Ps^{1}$. For the
Gauss and Kummer, we have 
\begin{align*}
\,_{2}F_{1}(a,b,c;x) & =\frac{\G(c)}{\G(a)\G(c-a)}\int_{0<u<1}u^{a-1}(1-u)^{c-a-1}(1-xu)^{-b}du,\\
\,_{1}F_{1}(a,c;x) & =\frac{\G(c)}{\G(a)\G(c-a)}\int_{0<u<1}e^{xu}u^{a-1}(1-u)^{c-a-1}du,
\end{align*}
which give the holomorphic solutions of the differential equations
at $x=0$ taking the value $1$ at this point, respectively. The integrals
for the Gauss family are understood as the Gelfand HGF (=Radon HGF
for $r=1$) on $\gras(2,4)$ corresponding to the partitions $(1,1,1,1),(2,1,1),(2,2),(3,1)$
and $(4)$, respectively. See \cite{Gelfand,Kimura-Haraoka} for the
detail. 

A Hermitian matrix integral analogue of the Gauss and its confluent
family is used and/or studied in several works \cite{Faraut,inamasu-ki,kimura-1,Kontsevich,Mehta,muirhead,muirhead-2}.
Let $\herm$ be the set of complex Hermitian matrices of size $r$,
which is a real vector space of dimension $r^{2}$. Then the analogue
of the Gauss family is 
\begin{align}
\text{Gauss: } & \int_{C}|U|^{a-r}|I-U|^{c-a-r}|I-UX|^{-b}\,dU,\nonumber \\
\text{Kummer: } & \int_{C}|U|^{a-r}|I-U|^{c-a-r}\etr(UX)\,dU,\nonumber \\
\text{Bessel: } & \int_{C}|U|^{c-r}\etr(UX-U^{-1})\,dU,\label{eq:intro-1}\\
\text{Hermite-Weber: } & \int_{C}|U|^{-c-r}\etr(UX-\frac{1}{2}U^{2})\,dU,\nonumber \\
\text{Airy: } & \int_{C}\etr(UX-\frac{1}{3}U^{3})\,dU,\nonumber 
\end{align}
where $X,U\in\herm$, $|U|:=\det U,\etr(U)=\expo(\Tr U)$ and $dU=\bigwedge_{i}dU_{i,i}\bigwedge_{i<j}d(\re U_{i,j})\wedge d(\im U_{i,j})$
is the Euclidean volume form on $\herm$. They can be understood as
the Radon HGF on $\gras(2r,4r)$ corresponding to the partitions $\lm$
of $4$: $(1,1,1,1),(2,1,1),(2,2),(3,1)$ and $(4)$, respectively
\cite{kimura-2}. Note that the number of parameters contained in
the above example is equal to $\ell(\lm)-1$, where $\ell(\lm)$ is
the length of $\lm$, namely the number of parts in $\lm$. For example,
Kummer's HGF contains $2(=\ell(\lm)-1)$ parameters $a,c$. However,
in the definition of Radon HGF corresponding to the above cases, we
know that the number of parameters contained in the Radon HGF is essentially
$3$ for any $\lm$. The reason for this gap concerning the number
of parameters can be explained by considering the action of continuous
part of the Weyl group on the Radon HGF (Proposition \ref{prop:act-3}).
On the other hand, the action of the part of finite group of $W_{\lm}$
gives the formulas for the Gauss HGF and Kummer's confluent HGF:
\begin{align}
\hyp 21(a,b,c;x) & =(1-x)^{-b}\hyp 21\left(c-a,b,c;\frac{x}{x-1}\right),\label{eq:intro-2}\\
\hyp 11(a,c;x) & =e^{x}\cdot\hyp 11\left(c-a,c;-x\right)\label{eq:intro-3}
\end{align}
\cite{Kimura-Koitabashi} and similar formulae for their Hermitian
matrix integral analogues (Propositions \ref{prop:ex-gauss-2}, \ref{prop:ex-kummer-1}).
So our motivation to study the Weyl group analogue associated with
the Radon HGF is to understand various transformation formulae known
for the classical HGF and its extension from a unified viewpoint.

This paper is organized as follows. In Section 2, we recall the definition
of the Radon HGF and the results necessary in this paper. In Section
3, we determine the structure of the normalizer $N_{G}(H_{\lm})$
(Theorem \ref{thm:weyl-main-1}) and give the explicit form of the
Weyl group $W_{\lm}$ associated with the Radon HGF (Proposition \ref{prop:weyl-group}).
They are the first main results of this paper. In Section 4, we study
the action of the Weyl group on the Radon HGF, which describes the
symmetry of the Radon HGF. This is the second main result of this
paper and is given in Theorem \ref{thm:main-2}. Section 5 is devoted
to the examples. We consider the Radon HGF on $\gras(2r,3r)$ and
on $\gras(2r,4r)$ associated with the partitions of 3 and 4, respectively.
For the Radon HGF on $\gras(2r,3r)$ for the partitions $(1,1,1),(2,1)$
and $(3)$ of $3$, we have the Hermitian matrix integral analogues
of the beta function, the gamma function and the Gaussian integral,
respectively. For the Radon HGF on $\gras(2r,4r)$ for the partitions
$(1,1,1,1),(2,1,1),(2,2),(3,1)$ and $(4)$, we have the Hermitian
matrix integral analogues of Gauss, Kummer, Bessel, Hermite-Weber
and Airy, respectively. For these cases we try to make clear what
Theorem \ref{thm:main-2} provides. We will see the reason why no
parameter is contained in the analogue of Airy function and how the
analogue of transformation formulae (\ref{eq:intro-2}), (\ref{eq:intro-3})
are obtained. The same formula is given as Proposition XV.3.4 in \cite{Faraut}
which is obtained by a different approach.

\section{Radon HGF}

\subsection{\label{subsec:Jordan-group}Jordan group}

We recall the definition of Radon HGF. For the detailed explanation,
see \cite{kimura-2}. Let $r$ and $N$ be positive integers such
that $r<N$ and assume $N=nr$ for some integer $n$. Suppose we are
given a partition $\lm=(n_{1},n_{2},\dots,n_{\ell})$ of $n$, namely
a sequence of positive integers $n_{1}\geq n_{2}\geq\cdots\geq n_{\ell}$
such that $|\lm|:=n_{1}+\cdots+n_{\ell}=n$. For $\lm$, let us consider
a complex Lie subgroup $H_{\lm}$ of the complex general linear group
$G=\GL N$. Put 

\[
J_{r}(p):=\left\{ h=\left(\begin{array}{cccc}
h_{0} & h_{1} & \dots & h_{p-1}\\
 & \ddots & \ddots & \vdots\\
 &  & \ddots & h_{1}\\
 &  &  & h_{0}
\end{array}\right)\mid h_{0}\in\GL r,\ h_{i}\in\mat(r)\right\} \subset\GL{pr},
\]
which is a Lie group called the (generalized) \emph{Jordan group}.
Define
\[
H_{\lm}:=\left\{ h=\diag(h^{(1)},\dots,h^{(\ell)})\mid h^{(j)}\in J_{r}(n_{j})\right\} \subset G.
\]
Then $H_{\lm}\simeq J_{r}(n_{1})\times\cdots\times J_{r}(n_{\ell})$,
where an element $(h^{(1)},\dots,h^{(\ell)})\in J_{r}(n_{1})\times\cdots\times J_{r}(n_{\ell})$
is identified with a block-diagonal matrix $\diag(h^{(1)},\dots,h^{(\ell)})\in H_{\lm}$.
In particular, for $\lm=(1,\dots,1)$, $H_{\lm}\simeq(\GL r)^{n}$
since $J_{r}(1)=\GL r$, and when $r=1$ it reduces to the Cartan
subgroup of $G$ consisting of diagonal matrices. We also use a unipotent
subgroup $\jro(p)\subset J_{r}(p)$: 
\[
\jro(p):=\left\{ h=\left(\begin{array}{cccc}
1_{r} & h_{1} & \dots & h_{p-1}\\
 & \ddots & \ddots & \vdots\\
 &  & \ddots & h_{1}\\
 &  &  & 1_{r}
\end{array}\right)\mid h_{i}\in\mat(r)\right\} .
\]
An element $h\in J_{r}(p)$ is expressed as $h=\sum_{0\leq i<p}h_{i}\otimes\La^{i}$
using the shift matrix $\La=(\de_{i+1,j})$ of size $p$. Taking into
account the expression $h=\sum_{0\leq i<p}h_{i}\otimes\La^{i}$, we
can describe $J_{r}(p)$ and $\jro(p)$ as follows. Put $R=\mat(r)$
and consider it as a $\C$-algebra. Let $R[w]$ be the ring of polynomials
in $w$ with coefficients in $R$. Then $J_{r}(p)$ is identified
with the group of units in the quotient ring of $R[w]$ by the principal
ideal $(w^{p})$:
\[
J_{r}(p)\simeq\left(R[w]/(w^{p})\right)^{\times}.
\]
Thus we can write $J_{r}(p)$ and $\jro(p)$ as 
\begin{align*}
J_{r}(p) & \simeq\left\{ h_{0}+\sum_{1\leq i<p}h_{i}w^{i}\in R[w]/(w^{p})\mid h_{0}\in\GL r\right\} ,\\
\jro(p) & \simeq\left\{ 1_{r}+\sum_{1\leq i<p}h_{i}w^{i}\in R[w]/(w^{p})\right\} .
\end{align*}
In the following we use freely this identification when necessary. 

\subsection{\label{subsec:Char-conf-1-1}Character of Jordan group}

In this section, we give the characters of the universal covering
group of Jordan group and of $H_{\lm}$. The Radon HGF is defined
as a Radon transform of these characters. The following lemma is easily
shown.
\begin{lem}
\label{lem:radon-conf-1}We have a group isomorphism 
\[
J_{r}(p)\simeq\GL r\ltimes\jro(p)
\]
defined by the correspondence $\GL r\ltimes\jro(p)\ni(g,h)\mapsto g\cdot h=\sum_{0\leq i<p}(gh_{i})\otimes\La^{i}\in J_{r}(p)$,
where the semi-direct product is defined by the action of $\GL r$
on $\jro(p)$: $h\mapsto g^{-1}hg=\sum_{i}(g^{-1}h_{i}g)\otimes\La^{i}$.
\end{lem}

Let us determine the characters of the universal covering group $\tilde{J}_{r}(p)$
of $J_{r}(p)$. Since $J_{r}(p)\simeq\GL r\ltimes\jro(p)$, it is
sufficient to determine characters of the universal covering group
$\widetilde{\Gl}(r)$ and of $\jro(p)$ which come from those of $\tilde{J}_{r}(p)$
by restriction. 

The characters of $\widetilde{\Gl}(r)$ is given as follows (Lemma
2.2 of \cite{kimura-2}). 
\begin{lem}
\label{lem:radon-conf-1-1-1} Any character $f:\widetilde{\Gl}(r)\to\cbatu$
is given by $f(x)=(\det x)^{a}$ for some $a\in\C$.
\end{lem}

Let us give the characters of $\jro(p)$. Let $\jroo(p)$ be the Lie
algebra of $\jro(p)$:
\[
\jroo(p)=\{X=\sum_{1\leq i<p}X_{i}w^{i}\mid X_{i}\in R\},
\]
where the Lie bracket of $X,Y\in\jroo(p)$ is given by $[X,Y]=\sum_{2\leq k<p}\sum_{i+j=k}[X_{i},Y_{j}]w^{k}$.
A character $\chi$ of $\jro(p)$ is obtained by lifting a character
of $\jroo(p)$ to that of $\jro(p)$ by the exponential map so that
the following diagram becomes commutative:

\[
\begin{CD}\jro(p)@>\chi>>\cbatu\\
@A\exp AA@AA\exp A\\
\jroo(p)@>d\chi>>\C.
\end{CD}
\]
Since $\jro(p)$ is a simply connected Lie group, the exponential
map
\[
\exp:\jroo(p)\to\jro(p),\,X\mapsto\exp(X)=\sum_{0\leq k}\frac{1}{k!}X^{k}=\sum_{0\leq k<p}\frac{1}{k!}X^{k}
\]
is a biholomorphic map. Hence we can consider the inverse map $\log:\jro(p)\to\jroo(p)$,
which, for $h=1_{r}+\sum_{1\leq i<p}h_{i}w^{i}\in\jro(p)$, defines
$\te_{k}(h)\in R$:
\begin{align}
\log h & =\log\left(1_{r}+h_{1}w+\cdots+h_{p-1}w^{p-1}\right)\nonumber \\
 & =\sum_{1\leq k<p}\frac{(-1)^{k+1}}{k}\left(h_{1}w+\cdots+h_{p-1}w^{p-1}\right)^{k}\nonumber \\
 & =\sum_{1\leq k<p}\theta_{k}(h)w^{k}.\label{eq:char-conf-0-1}
\end{align}
Here $\te_{k}(h)$ is a sum of monomials of noncommutative elements
$h_{1},\dots,h_{p-1}\in R$. If a weight of $h_{i}$ is defined to
be $i$, then the monomials appearing in $\te_{k}(h)$ has the weight
$k$. For example we have 
\begin{align*}
\te_{1}(h) & =h_{1},\\
\te_{2}(h) & =h_{2}-\frac{1}{2}h_{1}^{2},\\
\te_{3}(h) & =h_{3}-\frac{1}{2}(h_{1}h_{2}+h_{2}h_{1})+\frac{1}{3}h_{1}^{3},\\
\te_{4}(h) & =h_{4}-\frac{1}{2}(h_{1}h_{3}+h_{2}^{2}+h_{3}h_{1})+\frac{1}{3}(h_{1}^{2}h_{2}+h_{1}h_{2}h_{1}+h_{2}h_{1}^{2})-\frac{1}{4}h_{1}^{4}.
\end{align*}

\begin{lem}
\label{lem:Radon-conf-2-1-1}Let $\chi:\jro(p)\to\cbatu$ be a character
obtained from that of $\tilde{J}_{r}(p)$ by restricting it to $\jro(p)$.
Then there exists $\al=(\al_{1},\dots,\al_{p-1})\in\C^{p-1}$ such
that 
\begin{equation}
\chi(h;\al)=\exp\left(\sum_{1\leq i<p}\al_{i}\Tr\,\theta_{i}(h)\right).\label{eq:char-conf-1-1}
\end{equation}
Conversely, $\chi$ defined by (\ref{eq:char-conf-1-1}) gives a character
of $\jro(p)$.
\end{lem}

\begin{proof}
See Lemma 2.7 of \cite{kimura-2}.
\end{proof}
By virtue of the isomorphism in Lemma \ref{lem:radon-conf-1}, the
characters of $\tilde{J}_{r}(p)$ are determined as a consequence
of Lemmas \ref{lem:radon-conf-1-1-1} and \ref{lem:Radon-conf-2-1-1}.
\begin{prop}
\label{prop:Radon-conf-3-1}Any character $\chi_{p}:\tilde{J}_{r}(p)\to\cbatu$
is given by 
\[
\chi_{p}(h;\al)=(\det h_{0})^{\al_{0}}\exp\left(\sum_{1\leq i<p}\al_{i}\Tr\,\theta_{i}(\underline{h})\right),
\]
for some $\al=(\al_{0},\al_{1},\dots,a_{p-1})\in\C^{p}$, where $\underline{h}\in\jro(p)$
is defined by $h=\sum_{0\leq i<p}h_{i}w^{i}=\sum_{0\leq i<p}h_{0}(h_{0}^{-1}h_{i})w^{i}=h_{0}\cdot\underline{h}$.
\end{prop}

Now the characters of the group $\tilde{H}_{\lm}$ are given as follows.
\begin{prop}
\label{prop:char-conf-1}For a character $\chi_{\lm}:\tH_{\lm}\to\cbatu$,
there exists $\alpha=(\alpha^{(1)},\dots,\alpha^{(\ell)})\in\C^{n}$,
$\alpha^{(k)}=(\alpha_{0}^{(k)},\alpha_{1}^{(k)},\dots,\alpha_{n_{k}-1}^{(k)})\in\C^{n_{k}}$
such that 
\[
\chi_{\lm}(h;\al)=\prod_{1\leq k\leq\ell}\chi_{n_{k}}(h^{(k)};\al^{(k)}),\quad h=(h^{(1)},\cdots,h^{(\ell)})\in\tilde{H}_{\lm},\;h^{(k)}\in\tilde{J}_{r}(n_{k}).
\]
\end{prop}

\begin{cor}
In case $\lm=(1,\dots,1)$, a character $\chi:=\chi_{\lm}$ has the
form 
\[
\chi(h;\al)=\prod_{1\leq k\leq n}(\det h^{(k)})^{\al^{(k)}},\quad h=\diag(h^{(1)},\dots,h^{(n)}),\quad h^{(k)}\in\widetilde{\Gl}(r)
\]
with $\al=(\al^{(1)},\dots,\al^{(n)})\in\C^{n}$. 
\end{cor}

In the above corollary, we write $\al^{(k)}$ for $\al_{0}^{(k)}$
since $\al^{(k)}=(\al_{0}^{(k)})$ in the notation of Proposition
\ref{prop:char-conf-1}.

\subsection{\label{subsec:Definition-of-HGF}Definition of HGF of type $\protect\lm$}

To define the HGF as a Radon transform of the character $\chi_{\lm}:=\chi_{\lm}(\cdot;\al)$,
we prepare the space of independent variables of the HGF. Let $m$
be an integer such that $r<m<N$ and put $\mat'(m,N)=\{z\in\mat(m,N)\mid\rank\,z=m\}$.
According as the partition $\lm=(n_{1},\dots,n_{\ell})$ of $n$,
we write $z\in\mat'(m,N)$ as
\[
z=(z^{(1)},\dots,z^{(\ell)}),\;z^{(j)}=(z_{0}^{(j)},\dots,z_{n_{j}-1}^{(j)}),\;z_{k}^{(j)}\in\mat(m,r).
\]
Put
\begin{equation}
Z=\{z=(z^{(1)},\dots,z^{(\ell)})\in\mat'(m,N)\mid\rank z_{0}^{(k)}=r\;(1\leq k\leq\ell)\}.\label{eq:radon-00}
\end{equation}
Also we take $T=\gras(r,m)=\GL r\setminus\mat'(r,m)$ as the space
of integration variables. Denote by $t=(t_{a,b})\in\mat'(r,m)$ the
homogeneous coordinates of $T$ and by $[t]$ the point of $T$ with
the homogeneous coordinate $t$.

Let $h\in H_{\lm}$ be denoted as $h=\diag(h^{(1)},\dots,h^{(\ell)}),\quad h^{(j)}=\sum_{0\leq k<n_{j}}h_{k}^{(j)}\otimes\La^{k}\in J_{r}(n_{j})$.
Then define the map $\iota:\hlam\to\mat'(r,N)$ by the correspondence
\begin{equation}
h\mapsto(h_{0}^{(1)},\dots,h_{n_{1}-1}^{(1)},\dots,h_{0}^{(\ell)},\dots,h_{n_{\ell}-1}^{(\ell)}).\label{eq:radon-0}
\end{equation}
The map $\iota$ is injective and its image is a Zariski open subset
of $\mat'(r,N)$. The group $\hlam$ is sometimes identified with
the image 
\[
\iota(H_{\lm})=\{v=(v_{0}^{(1)},\dots,v_{n_{1}-1}^{(1)},\dots,v_{0}^{(\ell)},\dots,v_{n_{\ell}-1}^{(\ell)})\mid v_{k}^{(j)}\in\mat(r).\;\det v_{0}^{(j)}\neq0\;(\forall j,k)\}.
\]
This map is lifted naturally to the map $\tilde{H}_{\lm}\to\widetilde{\iota(H_{\lm})}$,
which will be denoted also by $\iota$. For $z\in Z$, put $tz=(tz^{(1)},\dots,tz^{(\ell)})\in\mat(r,N)$
and $tz^{(j)}=(tz_{0}^{(j)},\dots,tz_{n_{j}-1}^{(j)})\in\mat(r,n_{j}r)$,
where $tz_{k}^{(j)}\in\mat(r)$. Note that $\det(tz_{0}^{(j)})\neq0$
for generic $t\in\mat(r,m)$ since $\rank\,z_{0}^{(j)}=r$ by the
definition of $Z$. Then $\iota^{-1}(tz)$ can be considered, where
$tz^{(j)}$ is identified with $\sum_{0\leq k<n_{j}}tz_{k}^{(j)}\otimes\La^{k}\in J_{r}(n_{j})$
and $tz$ is identified with $\diag(tz^{(1)},\dots,tz^{(\ell)})\in H_{\lm}$.
Then we consider $\chi(\iota^{-1}(tz);\al)$. We write simply $\chi(tz;\al)$
for $\chi(\iota^{-1}(tz);\al)$ when there is no risk of confusion.

Assume here that the character $\chi_{\lm}(\cdot;\al)$ satisfies
the following condition.
\begin{assumption}
\label{assu:Radon-conf-4-1}(i) $\al_{0}^{(j)}\notin\Z$ for $1\leq j\leq\ell$,

(ii) $\al_{n_{j}-1}^{(j)}\neq0$ if $n_{j}\geq2$,

(iii) $\al_{0}^{(1)}+\cdots+\al_{0}^{(\ell)}=-m$.
\end{assumption}

By Assumption \ref{assu:Radon-conf-4-1} (iii), we see that $\chi_{\lm}(tz;\al)$
satisfies 
\begin{equation}
\chi_{\lm}((gt)z;\al)=(\det g)^{-m}\chi_{\lm}(tz;\al),\quad g\in\GL r,\label{eq:radon-1}
\end{equation}
which implies that $\chi_{\lm}(tz;\al)$ gives a multivalued analytic
section of the line bundle on $T$ associated with the character $\rho_{m}:\GL r\to\cbatu,\rho_{m}(g)=(\det g)^{m}$.
The branch locus of $\chi_{\lm}(tz;\al)$ on $T$ is 
\[
\bigcup_{1\leq j\leq\ell}S_{z}^{(j)},\quad S_{z}^{(j)}:=\{[t]\in T\mid\det(tz_{0}^{(j)})=0\}.
\]
Put $X_{z}:=T\setminus\cup_{1\leq j\leq\ell}S_{z}^{(j)}$, which is
a complement of the arrangement $\{S_{z}^{(1)},\dots,S_{z}^{(\ell)}\}$
of hypersurfaces of degree $r$ in $T$. 

We need $\tau(t)$, an $r(m-r)$-form in $t$-space, which can be
given as follows. For the homogeneous coordinates $t$ of $T$, put
$t=(t',t'')$ with $t'\in\mat(r),t''\in\mat(r,m-r)$ and consider
the affine neighbourhood $U=\{[t]\in T\mid\det t'\neq0\}$. Then we
can take affine coordinates $u$ of $U$ defined by $u=(t')^{-1}t''$.
Put $du:=\wedge_{i,j}du_{i,j}$, then we give $\tau(t)$ by
\begin{equation}
\tau(t)=(\det t')^{m}du.\label{eq:radon-2}
\end{equation}

\begin{example*}
In case $r=1$, $T=\gras(1,m)=\Ps^{m-1}$ with the homogeneous coordinates
$t=(t_{1},\dots,t_{m})$. We take $\tau(t)=\sum_{1\leq j\leq m}(-1)^{j+1}t_{j}dt_{1}\wedge\cdots\wedge\widehat{dt_{j}}\wedge\cdots\wedge dt_{m}$.
In the coordinate neighbourhood $U=\{[t]\in T\mid t_{1}\neq0\}$ with
the affine coordinates $(u_{2},\dots,u_{m})=(t_{2}/t_{1},\dots,t_{m}/t_{1})$,
we have 
\[
\tau(t)=t_{1}^{m}d\left(\frac{t_{2}}{t_{1}}\right)\wedge\cdots\wedge d\left(\frac{t_{m}}{t_{1}}\right)=t_{1}^{m}du_{2}\wedge\cdots\wedge du_{m}.
\]
\end{example*}
For $\tau(t)$ given by (\ref{eq:radon-2}), we have 
\begin{equation}
\tau(gt)=(\det g)^{m}\tau(t),\quad g\in\GL r.\label{eq:radon-3}
\end{equation}
Then, by virtue of (\ref{eq:radon-1}) and (\ref{eq:radon-3}), we
see that $\chi_{\lm}(tz;\al)\cdot\tau(t)$ gives a multivalued $r(m-r)$-form
on $X_{z}$. 
\begin{defn}
For a character $\chi_{\lm}(\cdot;\al)$ of the group $\tH_{\lm}$
satisfying Assumption \ref{assu:Radon-conf-4-1}, 
\begin{equation}
F_{\lm}(z,\al;C):=\int_{C(z)}\chi_{\lm}(tz;\al)\cdot\tau(t)\label{eq:radon-4}
\end{equation}
is called the Radon HGF of type $\lm$. Here $C(z)$ is an $r(m-r)$-cycle
of the homology group of locally finite chains $H_{r(m-r)}^{\Phi_{z}}(X_{z};\cL_{z})$
of $X_{z}$ with coefficients in the local system $\cL_{z}$ and with
the family of supports $\Phi_{z}$ determined by $\chi_{\lm}(tz;\al)$. 
\end{defn}

We briefly explain about the homology group $H_{r(m-r)}^{\Phi_{z}}(X_{z};\cL_{z})$.
For the detailed explanation, we refer to \cite{kimura-2} and references
therein. Write the integrand of (\ref{eq:radon-4}) as 
\[
\chi_{\lm}(tz;\al)=f(t,z)\exp(g(t,z)),
\]
where
\[
f(t,z)=\prod_{1\leq j\leq\ell}\left(\det tz_{0}^{(j)}\right)^{\al_{0}^{(j)}},\quad g(t,z)=\sum_{1\leq j\leq\ell}\sum_{1\leq k<n_{j}}\al_{k}^{(j)}\Tr\,\theta_{k}(\underline{tz}^{(j)}).
\]
Note that $f(t,z)\cdot\tau(t)$ concerns the multivalued nature of
the integrand whose ramification locus is $\cup_{j}S_{z}^{(j)}$.
On the other hand, $g(t,z)$ is a rational function on $T$ with a
pole divisor $\cup_{j;n_{j}\geq2}S_{z}^{(j)}$ and concerns the nature
of exponential increase to infinity or exponential decrease to zero
of the integrand when $[t]$ approaches to the pole divisor $\cup_{j;n_{j}\geq2}S_{z}^{(j)}$.
The monodromy of $f(t,z)\cdot\tau(t)$, which is the same as that
of $\chi_{\lm}(tz;\al)\cdot\tau(t)$, defines a rank one local system
$\cL_{z}$ on $X_{z}$. On the other hand, $g_{z}:=g|_{X_{z}}:X_{z}\to\C$
defines a family $\Phi_{z}$ of closed subsets of $X_{z}$ by the
condition
\[
A\in\Phi_{z}\iff A\cap g_{z}^{-1}(\{w\in\C\mid\mathrm{Re}\,w\geq a\})\ \ \mbox{is compact for any \ensuremath{a\in\R}}.
\]
Then $\Phi_{z}$ satisfies the condition of a family of supports \cite{kimura-2,Pham-1,Pham-2}
and we can consider a homology groups of locally finite chains with
coefficients in the local system $\cL_{z}$ whose supports belong
to $\Phi_{z}$. This homology group is denoted by $H_{\bullet}^{\Phi_{z}}(X_{z};\cL_{z})$.
Moreover there is a Zariski open subset $V\subset Z$ such that 

\[
\bigcup_{z\in V}H_{r(m-r)}^{\Phi_{z}}(X_{z};\cL_{z})\to V,
\]
which maps $H_{r(m-r)}^{\Phi_{z}}(X_{z};\cL_{z})$ to $z$, gives
a local system on $V$ \cite{kimura-2}. We take its local section
as $C=\{C(z)\}$ to obtain the Radon HGF of type $\lm$. 

We give an expression of $F_{\lm}$ in terms of the affine coordinates
$u=(u_{i,j})=(t')^{-1}t''$ of the chart $U=\{[t]\in T\mid\det t'\neq0\}$.
Using (\ref{eq:radon-1}) and (\ref{eq:radon-2}), we have 
\begin{align*}
F_{\lm}(z,\al;C) & =\int_{C(z)}\chi_{\lm}(\vec{u}z;\al)du\\
 & =\int_{C(z)}\prod_{1\leq j<\ell}\left(\det\vec{u}z_{0}^{(j)}\right)^{\al_{0}^{(j)}}\cdot\exp\left(\sum_{1\leq j\leq\ell}\sum_{1\leq k<n_{j}}\al_{k}^{(j)}\Tr\,\theta_{k}(\underline{\vec{u}z}^{(j)})\right)du,
\end{align*}
where $\vec{u}=(1_{r},u)$. In case $\lm=(1,\dots,1)$, the Radon
HGF is written as 
\[
F(z,\al;C)=\int_{C(z)}\prod_{1\leq j\leq n}\left(\det tz^{(j)}\right)^{\al^{(j)}}\tau(t)=\int_{C(z)}\prod_{1\leq j\leq n}\left(\det(\vec{u}z^{(j)})\right)^{\al^{(j)}}du
\]
and is said to be of \emph{non-confluent type}. 

We give an important property for the Radon HGF which states the covariance
of the function under the action of $\GL m\times\hlam$ on $Z$. we
see that the action
\[
\GL m\times\mnm\times\hlam\ni(g,z,h)\mapsto gzh\in\mnm
\]
induces that on the set $Z$. The following is Proposition 2.12 of
\cite{kimura-2}.
\begin{prop}
\label{prop:covariance}For the Radon HGF of type $\lm$, we have

(1) $F_{\lm}(gz,\al;C)=\det(g)^{-r}F_{\lm}(z,\al;\tilde{C}),\quad g\in\GL m,$

(2) $F_{\lm}(zh,\al;C)=F_{\lm}(z,\al;C)\chi_{\lm}(h;\al),\quad h\in\tilde{H}_{\lm}$.
\end{prop}

\section{\label{sec:Weyl-group}Weyl group analogue}

Recall that the Radon HGF of type $\lm$, $\lm$ is a partition of
$n$, is defined by the Radon transform of a character of the subgroup
$H_{\lm}\subset\GL N$, where $N=nr$. When $\lm=(1,\dots,1)$ and
$r=1$, $H_{\lm}$ reduces to the Cartan subgroup $H$ of $\GL N$
consisting of diagonal matrices. In this case $N_{\GL N}(H)/H$ is
the Weyl group of $\GL N$ in the usual sense. In this section, we
introduce an analogue of Weyl group taking $H_{\lm}$ instead of the
Cartan subgroup $H$, and we determine the structure of the Weyl group
analogue.

\subsection{Statement of the result}
\begin{defn}
For the subgroup $\hlam$ of $G=\GL N$ introduced in Section \ref{subsec:Jordan-group},
the Weyl group associated with $\hlam$ is defined by 
\[
W_{\lm}:=N_{G}(\hlam)/\hlam,
\]
where $N_{G}(\hlam)$ is the normalizer of $H_{\lm}$ in $G$:
\[
N_{G}(\hlam)=\{g\in G\mid ghg^{-1}\in H_{\lm}\text{ for }\forall h\in\hlam\}.
\]
\end{defn}

To determine the structure of the Weyl group for $H_{\lm}$, we change
the description of the partition $\lm$ of $n$ as 
\begin{equation}
\lm=(\overbrace{n_{1},\dots,n_{1}}^{p_{1}},\overbrace{n_{2},\dots,n_{2}}^{p_{2}},\dots,\overbrace{n_{s},\dots,n_{s}}^{p_{s}})\label{eq:weyl-1}
\end{equation}
with $n_{1}>n_{2}>\cdots>n_{s}>0$. Hence $p_{1}n_{1}+\cdots+p_{\ell}n_{\ell}=n$.
Accordingly, the group $H_{\lm}$ is written as 
\[
H_{\lm}=\prod_{1\leq i\leq s}H_{i},\quad H_{i}=\overbrace{J_{r}(n_{i})\times\cdots\times J_{r}(n_{i})}^{p_{i}\text{ times}}.
\]
 We also write $G_{i}=\GL{p_{i}n_{i}r}$ and regard $H_{i}$ as a
subgroup of $G_{i}$. 
\begin{prop}
\label{prop:weyl-2}We have the isomorphism
\[
\prod_{1\leq i\leq s}N_{G_{i}}(H_{i})\simeq N_{G}(H_{\lm})
\]
by the map
\[
\prod_{1\leq i\leq s}N_{G_{i}}(H_{i})\ni(X_{1},\dots,X_{s})\longmapsto\diag(X_{1},\dots,X_{s})\in\GL N.
\]
\end{prop}

\begin{proof}
Take $X\in N_{G}(H_{\lm})$ and write it as a block matrix $X=(X_{i,j})_{1\leq i,j\leq s},\;X_{i,j}\in\mat(p_{i}n_{i}r,p_{j}n_{j}r)$
according as the direct product structure $H_{\lm}=\prod_{i}H_{i}$.
We show that $X_{i,j}=0$ for $i\neq j$. In fact, take $A\in H_{\lm}$
such that $A=\diag(A_{1},\dots,A_{s}),$ $A_{i}\in H_{i}$ and each
$A_{i}$ has the form $A_{i}=\diag(A_{i,1},\dots,A_{i,p_{i}}),A_{i,k}\in J_{r}(n_{i})$,
where $A_{i,k}$ is a Jordan cell with eigenvalue $a_{i,k}$, and
furthermore suppose $\{a_{i,1},\dots,a_{i,p_{i}}\}\cap\{a_{j,1},\dots,a_{j,p_{j}}\}=\emptyset$
for $i\neq j$. Put $B=XAX^{-1}$. Since $X$ belongs to $N_{G}(H_{\lm})$,
$B$ is an element of $\hlam$. Write $B$ as $B=\diag(B_{1},\dots,B_{s})$
with $B_{i}\in H_{i}$. Since $A$ and $B$ are similar, their Jordan
normal forms coincide up to the ordering of Jordan cells. Taking into
account that $n_{1}>n_{2}>\cdots>n_{s}$, the only block which can
contain a Jordan cell of size $n_{1}r$ is $B_{1}$ in the block diagonal
expression of $B$. It follows that the Jordan normal form of $A_{1}$
and that of $B_{1}$ must coincide. In particular, their eigenvalues
coincide counting multiplicity. Next let us consider the Jordan cells
of size $n_{2}r$ for $A$ and $B$. Among $B_{2},\dots,B_{s}$, the
block which can contain the cells of size $n_{2}r$ is $B_{2}$. Hence
we see that the Jordan normal forms of $A_{2}$ and $B_{2}$ coincide.
In particular, their eigenvalues coincide counting multiplicity. Proceeding
in a similar way, we see that the eigenvalues of $A_{i}$ and $B_{i}$
coincide counting with multiplicity for any $i$. Write the relation
$XAX^{-1}=B$ in blockwise manner and get 
\begin{equation}
X_{i,j}A_{j}=B_{i}X_{i,j}\quad(1\leq i,j\leq s).\label{eq:weyl-1-2}
\end{equation}
We regard this relation as a linear equation for $X_{i,j}$. In case
$i\neq j$, $A_{i}$ and $B_{j}$ has no common eigenvalue. It follows
that equation (\ref{eq:weyl-1-2}) admits only trivial solution $X_{i,j}=0$.
Thus $X$ has the form $X=\diag(X_{1},\dots,X_{s}),X_{i}\in G_{i}$.
Now for any $A=\diag(A_{1},\dots,A_{s}),A_{i}\in H_{i}$, we have
\[
XAX^{-1}=\diag(X_{1}A_{1}X_{1}^{-1},\dots,X_{s}A_{s}X_{s}^{-1})\in H_{\lm}.
\]
Namely, for any $i$, we have $X_{i}A_{i}X_{i}^{-1}\in H_{i}$. This
implies that $X_{i}\in N_{G_{i}}(H_{i})$.
\end{proof}
\begin{prop}
\label{prop:weyl-3}(1) For $X\in N_{G_{i}}(H_{i})$, there exists
$\sm\in\S_{p_{i}}$ such that $X$ is uniquely written as 
\begin{equation}
X=\diag(X_{1},\dots,X_{p_{i}})\cdot P_{\sm},\quad X_{i}\in N_{\GL{n_{i}r}}(J_{r}(n_{i})),\label{eq:weyl-2}
\end{equation}
where $P_{\sm}=(\de_{j,\sm(k)}\cdot1_{n_{i}r})_{1\leq j,k\leq p_{i}}\in\GL{p_{i}n_{i}r}$
is a permutation matrix. 

(2) The correspondence $X\mapsto((X_{1},\dots,X_{p_{i}}),\sm)$ gives
an isomorphism
\begin{equation}
N_{\GL{p_{i}n_{i}r}}(H_{i})\simeq\left(\prod_{p_{i}\text{times}}N_{\GL{n_{i}r}}(J_{r}(n_{i}))\right)\rtimes\S_{p_{i}},\label{eq:weyl-3}
\end{equation}
where $\sm\in\S_{p_{i}}$ acts on $X=\diag(X_{1},\dots,X_{p_{i}})\in\prod_{p_{i}\text{times}}N_{\GL{n_{i}r}}(J_{r}(n_{i}))$
as a permutation of diagonal blocks:
\[
\diag(X_{1},\dots,X_{p_{i}})\mapsto\diag(X_{\sm(1)},\dots,X_{\sm(p_{i})}).
\]
\end{prop}

\begin{example*}
We give an example of $P_{\sm}$ in Proposition \ref{prop:weyl-3}
in the case $p_{i}=3$ and for $\sm\in\S_{3}$, $(\sm(1),\sm(2),\sm(3))=(2,3,1)$.
In this case 
\[
P_{\sm}=\left(\begin{array}{ccc}
0 & 0 & 1_{n_{i}r}\\
1_{n_{i}r} & 0 & 0\\
0 & 1_{n_{i}r} & 0
\end{array}\right),\quad X=\left(\begin{array}{ccc}
0 & 0 & X_{1}\\
X_{2} & 0 & 0\\
0 & X_{3} & 0
\end{array}\right).
\]
\end{example*}
\emph{Proof of Proposition \ref{prop:weyl-3}}. We prove assertion
(1). Recall that $G_{i}=\GL{p_{i}n_{i}r},\;H_{i}=J_{r}(n_{i})^{p_{i}}$.
Take $X\in N_{G_{i}}(H_{i})$. Write $X$ as a block matrix $X=(X_{j,k})_{1\leq j,k\leq p_{i}}$
with blocks $X_{j,k}\in\mat(n_{i}r)$. Since $X\in N_{G_{i}}(H_{i})$,
for any $A=\diag(A^{(1)},\dots,A^{(p_{i})})\in H_{i},\ A^{(j)}\in J_{r}(n_{i})$
there exists $B=\diag(B^{(1)},\dots,B^{(p_{i})})\in H_{i},\ B^{(j)}\in J_{r}(n_{i})$
such that 
\begin{equation}
XAX^{-1}=B.\label{eq:weyl-4}
\end{equation}
 Note that each $A^{(j)}$ has the form
\[
A^{(j)}=\sum_{0\leq k<n_{i}}A_{k}^{(j)}\otimes\La^{k},\qquad A_{k}^{(j)}\in\mat(r),
\]
where $\La$ is the shift matrix of size $n_{i}$. Assume that $A^{(j)}$
is a Jordan cell with an eigenvalue $a^{(j)}$ for any $j$ and that
$a^{(1)},\dots,a^{(p_{i})}$ are all distinct. The Jordan normal form
of $B$ is obtained by taking each of $B^{(1)},\dots,B^{(p_{i})}$
to the Jordan normal form. Since $A$ and $B$ are similar, $B$ has
the Jordan normal form with a Jordan cell of size $n_{i}r$. This
situation appears when one of $B^{(1)},\dots,B^{(p_{i})}$ is similar
to this Jordan cell. It follows that there exits $\sm\in\S_{p_{i}}$
such that 
\[
A^{(k)}\sim B^{(\sm(k))}\qquad(1\leq k\leq p_{i}).
\]
Then from (\ref{eq:weyl-4}), we have the equation
\begin{equation}
X_{j,k}A^{(k)}=B^{(j)}X_{j,k}\qquad(1\leq j,k\leq p_{i}).\label{eq:weyl-6}
\end{equation}
Since $A^{(k)}$ and $B^{(j)}$ have no common eigenvalue if $j\neq\sm(k)$,
we have $X_{j,k}=0$ in this case. Put $X_{j}=X_{j,\sm^{-1}(j)}$,
then we have 
\[
X=\diag(X_{1},\dots,X_{p_{i}})\cdot P_{\sm},
\]
where $P_{\sm}$ denotes the permutation matrix belonging to $\GL{p_{i}n_{i}r}$,
which, represented in a block matrix as for $X$, has the $(j,k)$-block
$\de_{j,\sm(k)}\cdot1_{n_{i}r}$. From (\ref{eq:weyl-6}), we have
$X_{j}A^{(\sm^{-1}(j))}X_{j}^{-1}=B^{(j)}\in J_{r}(n_{i})$ for any
$A=\diag(A^{(1)},\dots,A^{(p_{i})})\in H_{i}$. It follows that $X_{j}\in N_{\GL{n_{i}r}}(J_{r}(n_{i}))$.
Hence assertion (1) of the proposition is shown. Assertion (2) results
from the uniqueness of the representation (\ref{eq:weyl-2}). $\qquad\square$

\bigskip

As a consequence of Proposition \ref{prop:weyl-3}, we have the following
result in the non-confluent case. 
\begin{cor}
For the group $H=H_{(1,\dots,1)}$, we have 
\[
N_{G}(H)\simeq H\rtimes\S_{n}
\]
 and the Weyl group $W=W_{(1,\dots,1)}=N_{G}(H)/H\simeq\S_{n}$.
\end{cor}

The next step is to determine $N_{\GL{n_{i}r}}(J_{r}(n_{i}))$. To
state the results, we introduce the polynomials $\mu_{i,j}(x)$ of
$x=(x_{1},x_{2},\dots)$ indexed by $(i,j)\in\Z_{\geq0}\times\Z_{\geq0}$.
Consider the formal power series
\[
M(x,T)=x_{1}T+x_{2}T^{2}+\cdots
\]
in $T$ and define $\mu_{i,j}(x)$ in terms of generating functions:
\[
M(x,T)^{i}=\sum_{j\geq0}\mu_{i,j}(x)T^{j}\quad(i=0,1,2,\dots).
\]
Here we put $M(x,T)^{0}=1$ by definition. It implies
\[
\mu_{0,j}(x)=\begin{cases}
1 & \quad(j=0),\\
0 & \quad(j\geq1).
\end{cases}
\]
From the definition, we have
\begin{equation}
\mu_{i,j}(x)=\begin{cases}
0 & \quad(i>j),\\
\sum_{j_{1}+\cdots+j_{i}=j;j_{1},\dots,j_{i}\geq1}x_{j_{1}}\cdots x_{j_{i}} & \quad(i\leq j)
\end{cases}\label{eq:weyl-6-1}
\end{equation}
and therefore $\mu_{i,j}(x)$ is a polynomial of $x$ which is a sum
of monomials of degree at most $j$. In particular we have $\mu_{i,i}(x)=x_{1}^{i}$. 
\begin{lem}
\label{lem:weyl-6} (1) $\mu_{i,j}(x)=0$ for $i>j$.

(2) For any integer $i_{1},i_{2}\geq0$, we have 
\[
\mu_{i_{1}+i_{2},j}(x)=\sum_{k}\mu_{i_{1},k}(x)\mu_{i_{2},j-k}(x).
\]

(3) For the two sets of variables $x=(x_{1},x_{2},\dots),y=(y_{1},y_{2},\dots)$,
define $z=(z_{1},z_{2},\dots)$ by
\begin{equation}
z_{j}=\sum_{k}x_{k}\mu_{k,j}(y),\label{eq:weyl-7}
\end{equation}
then we have 
\begin{equation}
\mu_{i,j}(z)=\sum_{k}\mu_{i,k}(x)\mu_{k,j}(y).\label{eq:weyl-8}
\end{equation}
\end{lem}

\begin{proof}
Assertion (2) is a consequence of the trivial identity $M(x,T)^{i_{1}+i_{2}}=M(x,T)^{i_{1}}M(x,T)^{i_{2}}$.
It is sufficient to compare the coefficients of $T^{i_{1}+i_{2}}$
of the both sides in the expansion with respect to the indeterminate
$T$. We show assertion (3). Multiply $T^{j}$ to the both sides of
(\ref{eq:weyl-7}) and take a sum with respect to $j$, then we have
\begin{align*}
M(z,T) & =\sum_{j}z_{j}T^{j}=\sum_{j}\left(\sum_{k}x_{k}\mu_{k,j}(y)\right)T^{j}\\
 & =\sum_{k}x_{k}\left(\sum_{j}\mu_{k,j}(y)T^{j}\right)=\sum_{k}x_{k}M(y,T)^{k}\\
 & =M(x,M(y,T)).
\end{align*}
Hence there holds the identity $M(z,T)^{i}=M(x,M(y,T))^{i}$ for any
integer $i\geq0$. Expanding the both sides in the power series of
$T$ and comparing the coefficients of $T^{j}$, we have the identity
(\ref{eq:weyl-8}).
\end{proof}
For a positive integer $p$, put 
\begin{equation}
W_{r}(p)=\Bigl\{\mu(c)=\bigl(\mu_{i,j}(c)1_{r}\bigr)_{0\leq i,j<p}\mid\ c=(c_{1},c_{2},\dots,c_{p-1})\in\C^{p-1},c_{1}\neq0\Bigr\}.\label{eq:weyl-5}
\end{equation}
Note that $\mu(c)$ is a nonsingular block matrix of size $pr$ whose
$(i,j)$-block is the scalar matrix $\mu_{i,j}(c)1_{r}$ of size $r$
and that $\mu_{1,j}(c)=c_{j}$ for $j=1,\dots,p-1$.
\begin{lem}
$W_{r}(p)$ is a connected linear subgroup of $\GL{pr}$ of dimension
$p-1$.
\end{lem}

\begin{proof}
Looking at the form of the elements of $W_{r}(p)$, it is sufficient
to show the assertion in the case $r=1$. Put $W(p):=W_{1}(p)\subset\GL p$.
Let $(y_{i,j})$ be the coordinates of $\GL p$. Then $W(p)$ is defined
by the algebraic relations
\[
y_{i,j}-\mu_{i,j}(y_{1,1},\dots,y_{1,n-1})=0\quad(0\leq i,j<p)
\]
and hence it is a closed subset of $\GL p$. Let us see that $W(p)$
is a subgroup of $\GL p$. To this end, it is sufficient to show that,
for any $a=(a_{1},\dots,a_{p-1}),\,b=(b_{1},\dots,b_{p-1})\in\C^{p-1}$,
there exists $c=(c_{1},\dots,c_{p-1})$ such that $\mu(a)\mu(b)=\mu(c)$.
This condition is written using the entries of the matrix as
\begin{equation}
\sum_{0\leq k<p}\mu_{i,k}(a)\mu_{k,j}(b)=\mu_{i,j}(c)\quad(0\leq i,j<p).\label{eq:weyl-9}
\end{equation}
In particular, for $i=1$, we must have 
\[
\sum_{1\leq k\leq j}a_{k}\cdot\mu_{k,j}(b)=c_{j}\quad(1\leq j<p).
\]
So if we determine $c$ by this condition from $a,b$, we must show
the relation (\ref{eq:weyl-9}) holds. But this is just assertion
(3) of Lemma \ref{lem:weyl-6}.
\end{proof}
Now the structure of the normalizer of $H_{\lm}=\prod_{1\leq i\leq s}J_{r}(n_{i})^{p_{i}}$
for the partition $\lm$ in (\ref{eq:weyl-1}) is described as follows.
\begin{thm}
\label{thm:weyl-main-1}(1) For the subgroup $H_{\lm}=\prod_{1\leq i\leq s}J_{r}(n_{i})^{p_{i}}$
of $G=\GL N$, its normalizer can be expressed as 
\[
N_{G}(H_{\lm})=H_{\lm}\rtimes\left(\prod_{1\leq i\leq s}W_{r}(n_{i})^{p_{i}}\rtimes\cP_{i}\right),
\]
where $\cP_{i}$ is the group isomorphic to the permutation group
$\S_{p_{i}}$ consisting of the permutation matrices $P_{\sm}=(\de_{j,\sm(k)}\cdot1_{n_{i}r})_{1\leq j,k\leq p_{i}}$
of size $p_{i}n_{i}r$ for any $\sm\in\S_{p_{i}}$.

(2) $\cP_{i}$ acts on $\prod_{p_{i}\text{times}}W_{r}(n_{i})$ as
a permutation of diagonal blocks.
\end{thm}

\begin{prop}
\label{prop:weyl-group}The Weyl group associated with $\hlam$ is
given by
\[
W_{\lm}=N_{G}(H_{\lm})/H_{\lm}\simeq\prod_{1\leq i\leq s}W_{r}(n_{i})^{p_{i}}\rtimes\cP_{i}.
\]
\end{prop}

\subsection{Proof of Theorem \ref{thm:weyl-main-1}}

Taking account of Propositions \ref{prop:weyl-2} and \ref{prop:weyl-3},
it is sufficient to show the following proposition to complete the
proof of Theorem \ref{thm:weyl-main-1}. We write $n_{i}$ as $n$
in this subsection for the sake of simplicity of notation.
\begin{prop}
\label{prop:weyl-7} We have the isomorphism
\[
N_{\GL{nr}}(J_{r}(n))\simeq J_{r}(n)\rtimes W_{r}(n).
\]
Hence we have
\[
N_{\GL{nr}}(J_{r}(n))/J_{r}(n)\simeq W_{r}(n).
\]
\end{prop}

The proof of this proposition is reduced to showing Lemmas \ref{lem:weyl-8},
\ref{lem:weyl-9} and \ref{lem:weyl-10} below.
\begin{lem}
\label{lem:weyl-8}For $X\in N_{\GL{nr}}(J_{r}(n))$, there exists
$h:=\diag(h_{0},\dots,h_{0})\in J_{r}(n)$ such that $Y:=h^{-1}X=(Y_{i,j})_{0\leq i,j<n},Y_{i,j}\in\mat(r)$,
is a block upper triangular matrix whose diagonal blocks are diagonal
matrices of size $r$, namely, $Y_{i,j}=0\;(i>j)$ and $Y_{i,i}$
$(0\leq i<n)$ are diagonal matrices.
\end{lem}

\begin{proof}
Since $X$ is an element of $N_{\GL{nr}}(J_{r}(n))$, for any $A\in J_{r}(n)$,
$B:=XAX^{-1}$ is an element of $J_{r}(n)$ by definition. Put $B=\sum_{0\leq k<n}B_{k}\otimes\La^{k}$
with $B_{k}\in\mat(r)$. We take in particular $A=\sum_{0\leq k<n}A_{k}\otimes\La^{k}$
such that $A_{0}$ has distinct eigenvalues $a_{1},\dots,a_{r}$.
We assert that $A_{0}$ and $B_{0}$ are similar. In fact the characteristic
polynomials for $A$ and $B$ are $\det(xI-A_{0})^{n}$ and $\det(xI-B_{0})^{n}$,
respectively, and they coincide. It follows that $A_{0}$ and $B_{0}$
share the same eigenvalues. This means that $A_{0}$ and $B_{0}$
are similar. By this assertion we can take $h_{0}\in\GL r$ such that
$B_{0}=h_{0}A_{0}h_{0}^{-1}$. So define $h:=\diag(h_{0},\dots,h_{0})\in J_{r}(n).$Then
$B$ can be expressed as 
\[
B=\left(\begin{array}{cccc}
h_{0}A_{0}h_{0}^{-1} & B_{1} & \dots & B_{n-1}\\
 & \ddots & \ddots & \vdots\\
 &  & \ddots & B_{1}\\
 &  &  & h_{0}A_{0}h_{0}^{-1}
\end{array}\right)=h\left(\begin{array}{cccc}
A_{0} & B_{1}' & \dots & B_{n-1}'\\
 & \ddots & \ddots & \vdots\\
 &  & \ddots & B_{1}'\\
 &  &  & A_{0}
\end{array}\right)h^{-1},
\]
where $B_{k}'=h_{0}^{-1}B_{k}h_{0}$. So the relation $XAX^{-1}=B$
is written as 
\[
(h^{-1}X)A(h^{-1}X)^{-1}=A_{0}\otimes\La^{0}+\sum_{1\leq k<n}B_{k}'\otimes\La^{k}.
\]
 If we put $Y:=h^{-1}X$, which belongs to $N_{\GL{nr}}(J_{r}(n))$
because of $h\in J_{r}(n)$, it satisfies
\begin{equation}
YAY^{-1}=A_{0}\otimes\La^{0}+\sum_{1\leq k<n}B_{k}'\otimes\La^{k}.\label{eq:weyl-10}
\end{equation}
Next we shall show that $Y=(Y_{i,j})_{0\leq i,j<n},Y_{i,j}\in\mat(r)$,
is upper triangular blockwise, namely $Y_{i,j}=0$ for $i>j$, and
that the diagonal blocks $Y_{i,i}$ are diagonal matrices of size
$r$. Write (\ref{eq:weyl-10}) as 
\begin{equation}
YA=BY.\label{eq:weyl-11}
\end{equation}
We take $A=\sum_{0\leq k<n}A_{k}\otimes\La^{k}\in J_{r}(n)$ such
that $A_{0}$ has $r$ distinct eigenvalues and that $A_{1}$ is a
nonsingular diagonal matrix. Writing the $(i,j)$-block of both sides
of (\ref{eq:weyl-11}), we have 
\begin{align*}
(YA)_{i,j} & =Y_{i,0}A_{j}+Y_{i,1}A_{j-1}+\cdots+Y_{i,j}A_{0},\\
(BY)_{i,j} & =A_{0}Y_{i,j}+B_{1}'Y_{i+1,j}+\cdots+B_{n-1-i}'Y_{n-1,j}.
\end{align*}
We compare the blocks of the both sides of (\ref{eq:weyl-11}) in
the order indicated as 
\[
\begin{pmatrix}\\
\nwarrow\\
\nwarrow & \nwarrow\\
\nwarrow & \nwarrow & \nwarrow & \quad
\end{pmatrix},
\]
namely, first we compare the block located at the left lowest corner,
then we move to one upper diagonal array and we compare the lowest
block in it, then move to above in this array, and so on. Namely,
first we compare the $(i,j)$-block satisfying $i-j=n-1$. The possible
choice of $(i,j)$ is $(n-1,0)$. So we compare the $(n-1,0)$-block
of (\ref{eq:weyl-11}). Next we consider the $(i,j)$-blocks satisfying
$i-j=n-2$. The possible choice of $(i,j)$ is $(n-1,1)$ and $(n-2,0)$.
In this case, we compare firstly the $(n-1,1)$-block and then the
$(n-2,0)$-block of (\ref{eq:weyl-11}), and so on. Comparing the
$(n-1,0)$-block, we have
\begin{equation}
Y_{n-1,0}A_{0}=Y_{n-1,0}A_{0}.\label{eq:weyl-12}
\end{equation}
Since $A_{0}$ is a diagonal matrix with $r$ distinct eigenvalues,
it follows from (\ref{eq:weyl-12}) that $Y_{n-1,0}$ is also a diagonal
matrix. Next we consider the case $i-j=n-2$. Comparing the $(n-1,1)$-block,
we have 
\begin{equation}
Y_{n-1,0}A_{1}+Y_{n-1,1}A_{0}=A_{0}Y_{n-1,1}.\label{eq:weyl-13}
\end{equation}
Since $A_{1}$ and $Y_{n-1,0}$ are diagonal, so is $Y_{n-1,0}A_{1}$.
Comparing the off-diagonal entries of both sides and using the fact
that the eigenvalues of $A_{0}$ are all distinct, we see that $Y_{n-1,1}$
is also a diagonal matrix. It follows from (\ref{eq:weyl-13}) that
$Y_{n-1,0}A_{1}=0$ and $Y_{n-1,0}=0$ since $A_{1}$ is nonsingular.
Next we look at the $(n-2,0)$-block of (\ref{eq:weyl-11}):
\[
Y_{n-2,0}A_{0}=A_{0}Y_{n-2,0}+B_{1}'Y_{n-1,0}.
\]
Since $Y_{n-1,0}=0$ and $A_{0}$ has $r$ distinct eigenvalues, it
follows from the above relation that $Y_{n-2,0}$ is a diagonal matrix.
We turn to the $(i,j)$-block satisfying $i-j=n-3$. The possible
pair $(i,j)$ is $(n-1,2),(n-2,1)$ and $(n-3,0)$, and we apply the
similar reasoning for these cases in this order. From the $(n-1,2)$-block
of (\ref{eq:weyl-11}), we have 
\begin{equation}
Y_{n-1,0}A_{2}+Y_{n-1,1}A_{1}+Y_{n-1,2}A_{0}=A_{0}Y_{n-1,2}.\label{eq:weyl-14}
\end{equation}
Taking into account that $Y_{n-1,0}=0$ and $Y_{n-1,1},A_{1}$ are
diagonal, we compare the off-diagonal entries of both sides and we
see that $Y_{n-1,2}$ is a diagonal matrix. Then (\ref{eq:weyl-14})
reduces to $Y_{n-1,1}A_{1}=0$ and this leads to $Y_{n-1,1}=0$. For
the case $(i,j)=(n-2,1)$, relation (\ref{eq:weyl-11}) gives
\begin{equation}
Y_{n-2,0}A_{1}+Y_{n-2,1}A_{0}=A_{0}Y_{n-2,1}+B_{1}'Y_{n-1,1}.\label{eq:weyl-15}
\end{equation}
Since $Y_{n-1,1}=0$ and $Y_{n-2,0},A_{1}$ are diagonal, comparing
the off-diagonal entries of the both sides of (\ref{eq:weyl-15}),
we see that $Y_{n-2,1}$ is diagonal. Then (\ref{eq:weyl-15}) reduces
to $Y_{n-2,0}A_{1}=0$ and it leads to $Y_{n-2,0}=0$. For the case
$(i,j)=(n-3,0)$, relation (\ref{eq:weyl-11}) gives
\begin{equation}
Y_{n-3,0}A_{0}=A_{0}Y_{n-3,0}+B_{1}'Y_{n-2,0}+B_{2}'Y_{n-1,0}.\label{eq:weyl-16}
\end{equation}
Since $Y_{n-2,0}=Y_{n-1,0}=0$, (\ref{eq:weyl-16}) gives $Y_{n-3,0}A_{0}=A_{0}Y_{n-3,0}$,
from which we see that $Y_{n-3,0}$ is diagonal. Now we proceed inductively.
Let $k$ be an integer such that $0\leq k\leq n$ and suppose that
$Y_{i,j}=0$ for any $(i,j)$ satisfying $i-j>n-(k-1)$ and that $Y_{i,j}$
is a diagonal matrix for any $(i,j)$ satisfying $i-j=n-(k-1)$ as
the assumption of induction. We show that this assertion is valid
when $k$ is replaced by $k+1$. For $(i,j)$ satisfying $i-j=n-k$,
consider the condition (\ref{eq:weyl-11}):
\begin{equation}
(YA)_{i,j}=(BY)_{i,j}.\label{eq:weyl-18}
\end{equation}
In case $(i,j)=(n-1,k-1)$, (\ref{eq:weyl-18}) has the form
\[
Y_{n-1,0}A_{k-1}+Y_{n-1,1}A_{k-2}+\cdots+Y_{n-1,k-1}A_{0}=A_{0}Y_{n-1,k-1}.
\]
Noting $Y_{n-1,0}=Y_{n-1,1}=\cdots=Y_{n-1,k-3}=0$ holds by the assumption
of induction, it reduces to 
\begin{equation}
Y_{n-1,k-2}A_{1}+Y_{n-1,k-1}A_{0}=A_{0}Y_{n-1,k-1}.\label{eq:weyl-17}
\end{equation}
Moreover $Y_{n-1,k-2}$ is diagonal by the induction assumption, so
$Y_{n-1,k-2}A_{1}$ is a diagonal matrix. By comparing the off-diagonal
entries of (\ref{eq:weyl-17}), we see that $Y_{n-1,k-1}$ is diagonal
and then $Y_{n-1,k-2}=0$. Next we consider the case $(i,j)=(n-2,k-2)$.
In this case the condition (\ref{eq:weyl-18}) has the form
\begin{equation}
Y_{n-2,0}A_{k-2}+Y_{n-2,1}A_{k-3}+\cdots+Y_{n-2,k-3}A_{1}+Y_{n-2,k-2}A_{0}=A_{0}Y_{n-2,k-2}+B_{1}'Y_{n-1,k-2}.\label{eq:weyl-19}
\end{equation}
By the induction assumption, we have $Y_{n-2,0}=Y_{n-2,1}=\cdots=Y_{n-2,k-4}=0$.
Combining them with $Y_{n-1,k-2}=0$ which is verified above, we see
that (\ref{eq:weyl-19}) becomes
\[
Y_{n-2,k-3}A_{1}+Y_{n-2,k-2}A_{0}=A_{0}Y_{n-2,k-2}.
\]
Taking into account that $Y_{n-2,k-3}$ is diagonal, comparing the
off-diagonal entries of both sides, we see that $Y_{n-2,k-2}$ is
diagonal and as a consequence $Y_{n-2,k-3}=0$. The similar process
works for $(i,j)$ successively when $(i,j)$ is taken in the order
\[
(n-3,k-3)\to(n-4,k-4)\to\cdots\to(n-k+1,1),
\]
we see that in the order $Y_{n-3,k-4}\to Y_{n-4,k-5}\to\cdots\to Y_{n-k+1,0}$
these matrices become $0$ and that in the order $Y_{n-3,k-3}\to Y_{n-4,k-4}\to\cdots\to Y_{n-k+1,1}$
these matrices turn out to be diagonal. As the last step, consider
the case $(i,j)=(n-k,0)$. The condition (\ref{eq:weyl-18}) is 
\[
Y_{n-k,0}A_{0}=A_{0}Y_{n-k,0}+B_{1}'Y_{n-k+1,0}+B_{2}'Y_{n-k+2,0}+\cdots+B_{k-1}'Y_{n-1,0}.
\]
Since $Y_{n-k+1,0}=Y_{n-k+2,0}=\cdots=Y_{n-1,0}=0$ is already shown
in the induction process, it reduces to $Y_{n-k,0}A_{0}=A_{0}Y_{n-k,0}$
and this condition shows that $Y_{n-k,0}$ is a diagonal matrix. The
induction process works until $k$ reaches $n$.
\end{proof}
\begin{lem}
\label{lem:weyl-9}For $X\in N_{\GL{nr}}(J_{r}(n))$, there exists
$h\in J_{r}(n)$ such that 
\[
h^{-1}X=\begin{pmatrix}1_{r} & 0 & \dots & 0\\
 & c_{1,1}1_{r} & \dots & c_{1,n-1}1_{r}\\
 &  & \ddots & \vdots\\
 &  &  & c_{n-1,n-1}1_{r}
\end{pmatrix}\in N_{\GL{nr}}(J_{r}(n))
\]
 with some $c_{i,j}\in\C$.
\end{lem}

\begin{proof}
By Lemma \ref{lem:weyl-8}, we may assume $X=(X_{i,j})_{0\leq i,j<n}$
is blockwise upper triangular and the diagonal blocks $X_{0,0},X_{1,1},\dots,X_{n-1,n-1}$
are all diagonal matrices. Put $h'=(X_{0,0},\dots,X_{0,0})\in J_{r}(n)$
and $X'=(h')^{-1}X$. Then $X'$ has the form
\[
X'=\left(\begin{array}{cccc}
1_{r} & X_{0,1}' & \dots & X_{0,n-1}'\\
 & X_{1,1}' &  & X_{1,n-1}'\\
 &  & \ddots & \vdots\\
 &  &  & X_{n-1,n-1}'
\end{array}\right),\quad X_{i,j}'=X_{0,0}^{-1}X_{i,j}.
\]
Next we choose $(h'')^{-1}=1_{r}\otimes\La^{0}+\sum_{k=1}^{n-1}h_{k}\otimes\La^{k}$
appropriately so that the blocks in the $0$-th row of $Y=(h'')^{-1}X'$
is $(1_{r},0,\dots,0)$. The existence of such $h''$ is shown as
follows. Note that the $(0,j)$-block of $Y$ is 
\[
((h'')^{-1}X')_{0,j}=X_{0,j}'+h_{1}X_{1,j}'+\cdots+h_{j}X_{j,j}'
\]
and $X_{j,j}'$ is a nonsingular diagonal matrix. Then we can determine
$h_{1}$ by the condition $((h'')^{-1}X')_{0,1}=X_{0,1}'+h_{1}X_{1,1}'=0$
since $X_{1,1}'$ is nonsingular. Next we can determine $h_{2}$ by
the condition $((h'')^{-1}X')_{0,2}=X_{0,2}'+h_{1}X_{1,2}'+h_{2}X_{2,2}'=0$
since $X_{2,2}'$ is a nonsingular diagonal matrix. Inductively after
determining $h_{1},\dots,h_{j-1}$, we can determine $h_{j}$ by the
condition $((h'')^{-1}X')_{0,j}=0$. Now $Y\in N_{\GL{nr}}(J_{r}(n))$
is of the form
\[
Y=\left(\begin{array}{cccc}
1_{r} & 0 & \dots & 0\\
 & Y_{1,1} & \dots & Y_{1,n-1}\\
 &  & \ddots & \vdots\\
 &  &  & Y_{n-1,n-1}
\end{array}\right),\quad Y_{i,j}\in\mat(r),
\]
where the diagonal blocks $Y_{i,i}$ are diagonal matrices. We assert
that $Y_{i,j}$ are all scalar matrices. To show this fact, we use
the condition for $Y\in N_{\GL{nr}}(J_{r}(n))$, which implies that
for any $A=\sum_{0\leq k<n}A_{k}\otimes\La^{k}\in J_{r}(n)$, there
exists $B=\sum_{0\leq k<n}B_{k}\otimes\La^{k}\in J_{r}(n)$ such that
\begin{equation}
YA=BY.\label{eq:weyl-20}
\end{equation}
As the first step, we show that if $A=A_{0}\otimes\La^{0}=\diag(A_{0},\dots,A_{0})$,
then $A=B$. We compare the $(0,j)$-block of the both sides of (\ref{eq:weyl-20}),
namely $(YA)_{0,j}=(BY)_{0,j}$. Since 
\begin{align*}
(YA)_{0,j} & =\begin{cases}
A_{0} & (j=0),\\
0 & (j\geq1),
\end{cases}\\
(BY)_{0,j} & =\begin{cases}
B_{0} & (j=0),\\
B_{1}Y_{1,j}+\cdots+B_{j}Y_{j,j} & (j\geq1),
\end{cases}
\end{align*}
considering the case $j=0$, we have $A_{0}=B_{0}$. In the case $j=1$,
the condition is written as $0=B_{1}Y_{1,1}$, from which we see $B_{1}=0$
because $Y_{1,1}$ is nonsingular. Inductively we can conclude $B_{j}=0$
for $j=2,\dots n-1$ using the fact that $Y_{j,j}$ is a nonsingular
matrix and we have $A=B=\diag(A_{0},\dots,A_{0})$ for which (\ref{eq:weyl-20})
holds. This means
\[
Y_{i,j}A_{0}=A_{0}Y_{i,j}\qquad(i,j=1,\dots,n-1)
\]
holds for any $A_{0}\in\GL r$. It follows that $Y_{i,j}$ must be
a scalar matrix for any $1\leq i,j<n$. This proves the lemma.
\end{proof}
\begin{lem}
\label{lem:weyl-10}Let $Y\in N_{\GL{nr}}(J_{r}(n))$ be written as
$Y=(Y_{i,j}),\;Y_{i,j}\in\mat(r)$, in the form of block matrix and
every blocks $Y_{i,j}$ are scalar matrices. If the blocks in the
$0$-th row are $(Y_{0,0},Y_{0,1},\dots,Y_{0,n-1})=(1_{r},0,\dots,0)$.
Then there exists $(c_{1},\dots,c_{n-1})\in\C^{n-1},\;c_{1}\neq0$
such that $Y=(\mu_{i,j}(c)1_{r})_{0\leq i,j<n}$.
\end{lem}

\begin{proof}
By virtue of Lemma \ref{lem:weyl-9}, $Y$ can be written as $Y=1_{r}\otimes y$
with $y=(y_{i,j})_{0\leq i,j<n}\in N_{\GL n}(J_{1}(n))$ with $(y_{0,0},y_{0,1},\dots,y_{0,n-1})=(1,0,\dots,0)$.
Here we understand $1_{r}\otimes y$ is the block matrix whose $(i,j)$-block
is $y_{i,j}1_{r}$. Then applying Proposition 4.4 of \cite{Kimura-Koitabashi},
we get the conclusion of the lemma.
\end{proof}
Thus we have completed the proof of Proposition \ref{prop:weyl-7}
and hence the proof of Theorem \ref{thm:weyl-main-1}.
\begin{rem}
We can see that the Weyl group $W_{r}(n)$ for the Jordan group $J_{r}(n)$
is isomorphic to the automorphism group $\auto(S)$ of the algebra
$S=\C[T]/(T^{n})$. In fact, for a given $f\in\mathrm{Aut}(S)$, a
generator $T$ is taken to an another generator $T'=f(T)$ of $S$,
which is written as $T'=c_{1}T+\cdots+c_{n-1}T^{n-1}\ (=M(c,T))$
for some $c_{1},\dots,c_{n-1}\in\C$. Since $T'$ is a generator,
we must have $c_{1}\neq0$ and since $f$ is an algebra homomorphism,
$f$ induces the correspondence $T^{i}\mapsto M(c,T)^{i}$. 
\end{rem}

\section{Action of Weyl group on Radon HGF}

We study in this section the action of the normalizer $N_{G}(H_{\lm})$
on the Radon HGF of type $\lm.$ We adopt the notations in Section
\ref{sec:Weyl-group} for the partition $\lm$ of $n$ and for the
related subgroups of $G=\GL N$. Let $Z$ be the subset of $\mat'(m,N)$
defined by (\ref{eq:radon-00}) with respect to the group $H_{\lm}$
which is Zariski open in $\mat'(m,N)$, and let $Aut(Z)$ be the group
of holomorphic automorphisms of $Z$. By virtue of the explicit form
of $N_{G}(H_{\lm})$ given in Theorem \ref{thm:weyl-main-1}, it is
easily seen that the following lemma holds.
\begin{lem}
\label{lem:act-1}For \textup{$g\in N_{G}(H_{\lm})$ and }$z\in Z$\textup{,}
we have $zg\in Z$. Define for any $g\in N_{G}(H_{\lm})$ the map
$\f(g):Z\to Z$ by 
\[
\f(g)(z)=zg,\quad z\in Z.
\]
 Then we have the anti-homomorphism 
\[
\f:N_{G}(H_{\lm})\to Aut(Z).
\]
 In particular we have a representation of $W_{\lm}\subset N_{G}(H_{\lm})$
in the group $Aut(Z).$ 
\end{lem}

By Theorem \ref{thm:weyl-main-1}, we see that 
\[
W_{\lm}=U\rtimes\cP,
\]
 where 
\[
U:=\prod_{1\leq i\leq s}W_{r}(n_{i})^{p_{i}},\quad\cP:=\prod_{1\leq i\leq s}\cP_{i}.
\]
 Note that $U$ is the identity component of the Lie group $W_{\lm}$
and $\cP$ is the finite subgroup of $W_{\lm}$ isomorphic to $W_{\lm}/U\simeq\prod_{1\leq i\leq s}\cP_{i}$
with $\cP_{i}\simeq\frak{S}_{p_{i}}.$ As is seen from (\ref{eq:weyl-5})
and Theorem \ref{thm:weyl-main-1}, any element $g\in W_{\lm}$ can
be expressed uniquely in the form $g=(g_{a,b}'1_{r})$ with $g'=(g_{a,b}')_{1\leq a,b\leq n}\in\GL n$.
We define a homomorphism
\begin{equation}
\rho:W_{\lm}\to\GL n,\quad\rho(g)=g'.\label{eq:action-0-1}
\end{equation}

Let $\chi(\cdot\,;\al)$ be a character of the universal covering
group $\tilde{H}_{\lm}$: 
\[
\chi(\cdot\,;\al)=\prod_{1\leq i\leq s}\prod_{1\leq k\leq p_{i}}\chi_{n_{i}}(\cdot\,;\al^{(i,k)}),
\]
where
\begin{equation}
\al=(\al^{(1,1)},\dots,\al^{(1,p_{1})},\dots,\al^{(s,1)},\dots,\al^{(s,p_{s})})\in\C^{n},\quad\al^{(i,k)}=(\al_{0}^{(i,k)},\dots,\al_{n_{i}-1}^{(i,k)})\in\C^{n_{i}}\label{eq:action-0}
\end{equation}
and $\chi_{n_{i}}(\cdot\,;\al^{(i,k)})$ is a character of $\tilde{J}_{r}(n_{i})$
with the parameters $\al^{(i,k)}$. 
\begin{prop}
\label{prop:act-2}For $g\in W_{\lm}$, we have the identity 
\begin{equation}
\chi(\iota^{-1}(\iota(h)g);\al)=\chi(h;\al\cdot\,^{t}\rho(g))\quad\text{for}\ \ h\in\tilde{H}_{\lm},\label{eq:action-5}
\end{equation}
where $\iota$ is the map defined by (\ref{eq:radon-0}) and $\rho$
is that defined by (\ref{eq:action-0-1}). In particular, for $P_{\sm}\in\cP$
corresponding to $\sm\in\prod_{1\leq i\leq s}\S_{p_{i}}$, we have
\[
\chi(\iota^{-1}(\iota(h)P_{\sm});\sm(\al))=\chi(h;\al).
\]
\end{prop}

\begin{rem}
The map $h\mapsto\iota^{-1}(\iota(h)g)$ in (\ref{eq:action-5}) can
be written simply as $h\mapsto g^{-1}hg$ for $h\in\tilde{H}_{\lm}$
and $g\in W_{\lm}.$
\end{rem}

The proof of Proposition \ref{prop:act-2} will be given after stating
the results obtained from this proposition. We immediately see the
following result.
\begin{cor}
For $g\in W_{\lm}$ and $z\in Z$, we have 
\begin{equation}
\chi(\iota^{-1}(tzg);\al)=\chi(\iota^{-1}(tz);\al\cdot\,^{t}\rho(g)).\label{eq:action-1}
\end{equation}
 In particular, for $g\in\cP,$ we have 
\begin{equation}
\chi(\iota^{-1}(tzg);\al\cdot\rho(g))=\chi(\iota^{-1}(tz);\al).\label{eq:action-2}
\end{equation}
 
\end{cor}

Integrating the relations (\ref{eq:action-1}) and (\ref{eq:action-2})
on the same cycle, we get the following result.
\begin{thm}
\label{thm:main-2}Let $F(z,\al;C)$ be the Radon HGF of type $\lm.$
Then we have the following transformation formulae.

(1) For $g\in W_{\lm},$ we have 
\[
F(zg,\al;C)=F(z,\al\cdot\,^{t}\rho(g);C).
\]

(2) For $P_{\sm}\in\cP$ corresponding $\sm\in\prod_{1\leq i\leq s}\S_{p_{i}}$,
we have 
\begin{equation}
F(zP_{\sm},\sm(\al);C)=F(z,\al;C).\label{eq:action-8}
\end{equation}
 
\end{thm}

The following result is a consequence of assertion (1) of Theorem
\ref{thm:main-2}, which asserts that by the action of the continuous
part $U$ of the Weyl group $W_{\lm}$, the parameter $\al^{(i,k)}=(\al_{0}^{(i,k)},\dots,\al_{n_{i}-1}^{(i,k)})$
in $F_{\lm}(z,\al;C)$ can be taken to $(\al_{0}^{(i,k)},0,\dots,0,a)$,
where $a$ is an arbitrary nonzero complex number.
\begin{prop}
\label{prop:act-3} Let $\al\in\C^{n}$ be the parameter for the Radon
HGF $F(z,\al;C)$ expressed as in (\ref{eq:action-0}). For any parameter
\[
\be=(\be{}^{(1,1)},\dots,\be^{(1,p_{1})},\dots,\be^{(s,1)},\dots,\be^{(s,p_{s})})\in\C^{n},\quad\be^{(i,k)}=(\be_{0}^{(i,k)},\dots,\be_{n_{i}-1}^{(i,k)})\in\C^{n_{i}}
\]
satisfying Assumption \ref{assu:Radon-conf-4-1} and $\al{}_{0}^{(i,k)}=\be{}_{0}^{(i,k)}\;(1\leq i\leq s,1\leq k\le p_{i})$,
there exists $g\in U$ such that the change of variables $z\mapsto z'=zg^{-1}$
transforms $F(z,\al;C)$ to $F(z',\be;C)$. In particular we can take
$\be$ as 
\[
\be{}^{(i,k)}=(\al{}_{0}^{(i,k)},0,\dots,0,1).
\]
\end{prop}

\begin{proof}
The identity component $U=\prod_{1\leq i\leq s}W_{r}(n_{i})^{p_{i}}$
of $W_{\lm}$ acts on the space $\C^{n}$ of parameters $\al$ blockwise,
namely for $g\in U$, $\rho(g)\in\GL n$ defined by (\ref{eq:action-0-1})
has the form 
\[
\rho(g)=\diag(\rho(g)^{(1,1)},\dots,\rho(g)^{(1,p_{1})},\dots,\rho(g)^{(s,1)},\dots,\rho(g)^{(s,p_{s})}),\quad\rho(g)^{(i,k)}\in\GL{n_{i}}
\]
and it acts on the $(i,k)$-block $\al^{(i,k)}\in\C^{n_{i}}$ of $\al$
by
\[
\al^{(i,k)}\mapsto\al^{(i,k)}\cdot\,{}^{t}\rho(g)^{(i,k)}.
\]
So it is sufficient to prove the assertion in the particular case
$\lm=(n).$ We show that, for any $\al=(\al_{0},\dots,\al_{n-1})\in\C^{n}$
satisfying $\al_{n-1}\neq0$, there exists $g\in W_{r}(n)$ such that
\[
\be=\al\cdot{}^{t}\rho(g)=(\al_{0},0,\dots,0,1).
\]
Recall that $g\in W_{r}(n)$ has the form $g=(\mu_{i,j}(c)\cdot1_{r})_{0\leq i,j<n}\in\GL{rn}$,
where we can take $c=(c_{1},\dots,c_{n-1})\in\C^{n-1}$ arbitrarily
under the condition $c_{1}\neq0$ and $\rho(g)=(\mu_{i,j}(c))_{0\leq i,j<n}\in\GL n$.
From assertion (1) of Lemma \ref{lem:weyl-6}, we have 
\begin{equation}
\be_{i}=(\al\cdot{}^{t}\rho(g))_{i}=\sum_{0\leq j<n}\al_{j}\mu_{i,j}(c)=\al_{i}\mu_{i,i}(c)+\cdots+\al_{n-1}\mu_{i,n-1}(c).\label{eq:action-3}
\end{equation}
Consider (\ref{eq:action-3}) in case $i=n-1.$ Then $\be{}_{n-1}=\al_{n-1}\mu_{n-1,n-1}(c)$.
Noting that $\mu_{n-1,n-1}=c_{1}^{n-1}$ and $\al_{n-1}\neq1,$ we
can choose $c_{1}\neq0$ so that $\be{}_{n-1}=1.$ Next we consider
(\ref{eq:action-3}) for $i=n-2$ and we have 
\begin{equation}
\be{}_{n-2}=\al_{n-2}\mu_{n-2,n-2}(c)+\al_{n-1}\mu_{n-2,n-1}(c).\label{eq:action-4}
\end{equation}
 Note that we see from (\ref{eq:weyl-6-1}), the terms $\mu_{n-2,n-2}$
and $\mu_{n-2,n-1}$ has the form 
\[
\mu_{n-2,n-2}=c_{1}^{n-2},\ \ \mu_{n-2,n-1}=(n-2)c_{1}^{n-3}c_{2}.
\]
Using the condition $\al_{n-1}\neq0,$ we can determine $c_{2}$ so
that the right hand side of (\ref{eq:action-4}) becomes $0$. Proceeding
in inductive manner, we can choose $c_{3},\dots,c_{n-1}$ so that
$\be{}_{n-3},\dots,\be{}_{1}$ become all zero. Lastly from condition
(\ref{eq:action-3}) for $i=0$, we have $\be{}_{0}=\al_{0}$ because
$\mu_{0,j}(c)=\delta_{0,j}$ by  definition. 
\end{proof}
The rest of this section is devoted to the proof of Proposition \ref{prop:act-2}.
First we prove the proposition for $g\in U.$ Since $g$ acts on $H_{\lm}$
and on $\C^{n}$ blockwise as explained in the proof Proposition \ref{prop:act-3},
it is sufficient to prove the proposition for the special case where
$\lm=(n)$, namely $H_{\lm}=J_{r}(n)$ and $\chi$ is a character
of $\tilde{J}_{r}(n).$
\begin{lem}
Assume $\lm=(n)$ and $H_{\lm}=J_{r}(n)$. Then the identity (\ref{eq:action-5})
holds for any $g\in W_{r}(n).$ 
\end{lem}

\begin{proof}
Take $g\in W_{r}(n).$ For $h\in J_{r}(n),$ put 
\[
h'=\iota^{-1}(\iota(h)g).
\]
 Then we have 
\begin{align*}
\log\chi(h';\al) & =\sum_{0\leq i<n}\al_{i}\Tr\,\theta_{i}(h')\\
 & =(\Tr\,\theta_{0}(h'),\dots,\Tr\,\theta_{n-1}(h'))\left(\begin{array}{c}
\al_{0}\\
\vdots\\
\al_{n-1}
\end{array}\right)
\end{align*}
 Set $\Xi(h):=(\theta_{0}(h),\dots,\theta_{n-1}(h))\in\mat(r,nr)$
and assert that the following identity holds. 
\begin{equation}
\Xi(h')=\Xi(h)g.\label{eq:action-7}
\end{equation}
 In fact, by the definition of the function $\theta_{i}(h)$, we have
\begin{align}
\exp\left(\sum_{i}\theta_{i}(h')T^{i}\right) & \equiv h_{0}'+h_{1}'T+\cdots+h_{n-1}'T^{n-1}\quad\text{mod. \ensuremath{T^{n}}}\nonumber \\
 & =\iota(h)g\left(\begin{array}{c}
1_{r}\\
T\cdot1_{r}\\
\vdots\\
T^{n-1}\cdot1_{r}
\end{array}\right).\label{eq:action-6}
\end{align}
Since $g\in W_{r}(n)$ has the form $g=(\mu_{i,j}(x)\cdot1_{r})_{1\leq i,j<n}$
with some $x=(x_{1},\dots,x_{n-1})\in\C^{n-1},$ from the definition
$\mu_{i,j}(x),$ we have 
\[
g\left(\begin{array}{c}
1_{r}\\
T\cdot1_{r}\\
\vdots\\
T^{n-1}\cdot1_{r}
\end{array}\right)\equiv\left(\begin{array}{c}
1_{r}\\
f(x,T)\cdot1_{r}\\
\vdots\\
f(x,T)^{n-1}\cdot1_{r}
\end{array}\right)\quad\text{mod. \ensuremath{T^{n}}},
\]
where $f(x,T)=x_{1}T+\cdots+x_{n-1}T^{n-1}.$ Therefore the right
hand side of (\ref{eq:action-6}) equals $\exp\left(\sum_{i}\theta_{i}(h)f^{i}\right)$
modulo $T^{n}.$ It follows that 
\[
\Xi(h')\left(\begin{array}{c}
1_{r}\\
T\cdot1_{r}\\
\vdots\\
T^{n-1}\cdot1_{r}
\end{array}\right)\equiv\Xi(h)\left(\begin{array}{c}
1_{r}\\
f(x,T)\cdot1_{r}\\
\vdots\\
f(x,T)^{n-1}\cdot1_{r}
\end{array}\right)\equiv\Xi(h)g\left(\begin{array}{c}
1_{r}\\
T\cdot1_{r}\\
\vdots\\
T^{n-1}\cdot1_{r}
\end{array}\right)\quad\text{mod.}\ \ T^{n}.
\]
Thus we have the identity (\ref{eq:action-7}). Now the identity (\ref{eq:action-5})
is immediate. In fact, 
\begin{align*}
\log\chi(\iota^{-1}(\iota(h)g);\al) & =(\Tr\,\theta_{0}(h'),\dots,\Tr\,\theta_{n-1}(h')){}^{t}\al\\
 & =(\Tr\,\theta_{0}(h),\dots,\Tr\,\theta_{n-1}(h))\rho(g)\cdot{}^{t}\al\\
 & =(\Tr\,\theta_{0}(h),\dots,\Tr\,\theta_{n-1}(h)){}^{t}(\al\,^{t}\rho(g))\\
 & =\log\chi(h;\al\,^{t}\rho(g))
\end{align*}
by virtue of (\ref{eq:action-7}) and the specific form $g=(\mu_{i,j}(x)\cdot1_{r})_{1\leq i,j<n}$
of $g$, where each $(i,j)$-block is a scalar matrix $\mu_{i,j}(x)\cdot1_{r}$.
Exponentiating this identity, we get the desired identity (\ref{eq:action-5})
in this particular case. 
\end{proof}
Next we want to prove (\ref{eq:action-5}) for $g\in\cP.$ Taking
account of the structure of the group $\cP=\prod_{1\leq i\leq s}\cP_{i},$
to show the identity (\ref{eq:action-5}) for $g\in\cP,$ it is enough
to show it for each $\cP_{i}.$ Therefore it will be sufficient to
consider the case that the partition $\lm$ of $n$ has the form $\lm=(l,\dots,l)$
with the length $p$, namely $lp=n,$ and $\cP\simeq\frak{\S}_{p}$
is the subgroup of $\GL{rn}$ consisting of permutation matrices $P$
which is written blockwise as $P=(P_{j,k})_{1\leq j,k\leq p},P_{j,k}\in\mat(r)$
such that $P_{j,k}=0$ or $P_{j,k}=1_{rl}.$ 
\begin{lem}
The identity (\ref{eq:action-5}) holds for the case $H_{\lm}=\prod_{p\,\text{times}}J_{r}(l)\subset\GL{rn}$
and $g\in\cP\simeq\frak{\S}_{p}$. 
\end{lem}

\begin{proof}
Let $g\in\cP$ be the block permutation matrix $P_{\sm}=(\delta_{i,\sigma(j)}\cdot1_{rl})$
corresponding to $\sm\in\S_{p}$. Then $\rho(g)=(\delta_{i,\sigma(j)}\cdot1_{l})\in\GL n$
is also a block permutation matrix for $\sm\in\S_{p}$. For $h=\diag(h^{(1)},\dots,h^{(p)})\in H_{\lm},h^{(j)}\in J_{r}(l)$
and for $\al=(\al^{(1)},\dots,\al^{(p)})\in\C^{n},\al^{(j)}\in\C^{l}$,
we have 
\begin{align*}
\iota(h)g & =(h^{(1)},\dots,h^{(p)})P_{\sm}=(h^{(\sm(1))},\dots,h^{(\sm(p))}),\\
\al\cdot\rho(g) & =(\al^{(1)},\dots,\al^{(p)})\rho(g)=(\al^{(\sm(1))},\dots,\al^{(\sm(p))}).
\end{align*}
Then 
\begin{align}
\chi_{\lm}(\iota^{-1}(\iota(h)g);\al\cdot\rho(g)) & =\prod_{1\leq k\leq p}\chi_{l}(h^{(\sm(k))};\al^{(\sm(k))})\nonumber \\
 & =\prod_{1\leq k\leq p}\chi_{l}(h^{(k)};\al^{(k)})\label{eq:action-9}\\
 & =\chi_{\lm}(h;\al).\nonumber 
\end{align}
Since $\rho(g)$ is a permutation matrix and hence in particular an
orthogonal matrix, we have $^{t}\rho(g)=\rho(g)^{-1}$. Then the desired
identity (\ref{eq:action-5}) follows from (\ref{eq:action-9}). 
\end{proof}
Now the proof of Proposition \ref{prop:act-2} is already completed
since any element of $W_{\lm}$ is a product of those of $U$ and
$\cP$.

\section{Examples}

In this section, we consider the Radon HGF corresponding to classical
HGFs and explain what Theorem \ref{thm:main-2} implies when it is
applied to the examples. As classical HGFs, we consider here the beta
and gamma functions and the Gaussian integral as the first group,
the second group is the Gauss HGF and its confluent family, namely
Kummer's confluent HGF, Bessel function, Hermite-Weber function and
Airy function. For these classical HGFs, we know their Hermitian matrix
integral analogues. As is explained in \cite{kimura-2}, they can
be understood as particular cases of the Radon HGF. We will apply
Theorem \ref{thm:main-2} to these Hermitian matrix integrals. 

To make a link between the Hermitian matrix integral analogues and
the Radon HGF, we consider the Radon HGF of type $\lm$ in the case
$m=2r,N=nr$ and introduce a Zariski open subset $Z_{\lm}\subset\mat'(2r,nr)$
assuming some additional condition on $z\in\mat'(2r,nr)$, which is
considered as the space of independent variables for the Radon HGF.
Note that $n\geq3$ since $N>m$ by assumption. 

Let a partition $\lm=(n_{1},\dots,n_{\ell})$ of $n$ be given. Note
that a partition is identified with a Young diagram. We say that $\mu=(m_{1},\dots,m_{\ell})\in\Z_{\geq0}^{\ell}$
is a subdiagram of $\lm$ if it satisfies $0\leq m_{k}\leq n_{k}\quad(\forall k)$
and we write $\mu\subset\lm$. The weight of $\mu$ is defined by
$|\mu|:=m_{1}+\cdots+m_{\ell}$. Let $\mu$ be such that $\mu\subset\lm$
and $|\mu|=2$. Then $\mu$ has the form either 
\begin{equation}
\mu=(0,\dots,0,\overset{i}{1},0,\dots,0,\overset{j}{1},0,\dots,0)\;\text{or }\mu=(0,\dots,0,\overset{i}{2},0,\dots,0).\label{eq:exa-1}
\end{equation}
The first case means that $m_{i}=m_{j}=1$ and $m_{k}=0$ for $k\neq i,j$,
and the second case means that $m_{j}=2$ and $m_{k}=0$ for $k\neq j$.
Using this notation we define a Zariski open subset $Z_{\lm}\subset\mat'(2r,nr)$
as follows. According as $\lm=(n_{1},\dots,n_{\ell})$, write $z\in\mat'(2r,nr)$
as
\[
z=(z^{(1)},\dots,z^{(\ell)}),\quad z^{(j)}=(z_{0}^{(j)},\dots,z_{n_{j}-1}^{(j)}),\quad z_{k}^{(j)}\in\mat(2r,r).
\]
Take a subdiagram $\mu\subset\lm$ with $|\mu|=2$. Then according
as the form of $\mu$ given in (\ref{eq:exa-1}), we put
\[
z_{\mu}=(z_{0}^{(i)},z_{0}^{(j)})\;\text{or }z_{\mu}=(z_{0}^{(i)},z_{1}^{(i)}),
\]
respectively. Then $Z_{\lm}$ is defined as 
\[
Z_{\lm}:=\{z\in\mat(2r,nr)\mid\det z_{\mu}\neq0\;\text{for any}\;\mu\subset\lm,|\mu|=2\}.
\]
It is easily seen that $Z_{\lm}$ is invariant by the action $\GL{2r}\curvearrowright\mat'(2r,nr)\curvearrowleft H_{\lm}$.
Taking into account the covariance property for the Radon HGF with
respect to the action of $\GL{2r}\times H_{\lm}$ given in Proposition
\ref{prop:covariance}, we try to take the independent variable $z$
to a simpler form $\bx\in Z_{\lm}$ which gives a representative of
the orbit $O(z)$ of $z$.

\subsection{Analogues of the beta and gamma functions}

The classical beta and gamma functions are defined as 

\begin{align*}
B(a,b) & =\int_{0<u<1}u^{a-1}(1-u)^{b-1}du,\\
\G(a) & =\int_{u>0}e^{-u}u^{a-1}du
\end{align*}
and their Hermitian matrix integral analogues are 
\begin{align}
B_{r}(a,b) & =\int_{0<U<1}|U|^{a-r}|1_{r}-U|^{b-r}dU,\label{eq:exa-2}\\
\G_{r}(a) & =\int_{U>0}\etr(-U)|U|^{a-r}dU,\label{eq:exa-3}
\end{align}
where $U$ is an integration variable belonging to the set $\herm$
of Hermitian matrices of size $r$, $|U|:=\det U$, $\etr(U)=\exp(\Tr\,U)$,
$U>0$ and $1_{r}-U>0$ mean that the Hermitian matrix $U$ and $1_{r}-U$
are positive definite, respectively, and $dU$ is the Euclidean volume
form 
\begin{equation}
dU=dU_{1,1}\wedge\cdots\wedge dU_{r,r}\bigwedge_{i<j}d(\re U_{i,j})\wedge d(\im U_{i,j}).\label{eq:exa-4}
\end{equation}

To understand these matrix integrals as particular cases of the Radon
HGF, we consider the Radon HGF in the case where the partitions $\lm$
are $(1,1,1),(2,1)$ and $(3)$ of the weight $3$ and the space $Z_{\lm}$
of independent variables is a subset of $\mat(2r,3r)$. So we consider
the following subgroup $\hlam$ of $\GL{3r}$:

\begin{align*}
H_{(1,1,1)} & =\left\{ \left(\begin{array}{ccc}
h_{0}\\
 & h_{1}\\
 &  & h_{2}
\end{array}\right)\right\} ,\;H_{(2,1)}=\left\{ \left(\begin{array}{ccc}
h_{0} & h_{1}\\
 & h_{0}\\
 &  & h_{2}
\end{array}\right)\right\} ,\\
H_{(3)} & =\left\{ \left(\begin{array}{ccc}
h_{0} & h_{1} & h_{2}\\
 & h_{0} & h_{1}\\
 &  & h_{0}
\end{array}\right)\right\} ,
\end{align*}
where $h_{k}\in\mat(r)$. In the rest of this section, we use different
notations from that used in Sections \ref{subsec:Jordan-group}, \ref{subsec:Char-conf-1-1}
and \ref{subsec:Definition-of-HGF} about indices in order to avoid
unnecessary complexity of notations. We use the same convention for
the parameter $\al$ in the character of the universal covering group
$\tilde{H}_{\lm}$:
\begin{align*}
\chi_{(1,1,1)} & (h;\al)=(\det h_{0})^{\al_{0}}(\det h_{1})^{\al_{1}}(\det h_{2})^{\al_{2}},\\
\chi_{(2,1)} & (h;\al)=(\det h_{0})^{\al_{0}}\etr(\al_{1}h_{0}^{-1}h_{1})(\det h_{2})^{\al_{2}},\\
\chi_{(3)} & (h;\al)=(\det h_{0})^{\al_{0}}\etr\left(\al_{1}h_{0}^{-1}h_{1}+\al_{2}\left(h_{0}^{-1}h_{2}-\frac{1}{2}(h_{0}^{-1}h_{1})^{2}\right)\right).
\end{align*}
The spaces on which the Radon HGFs are defined are 
\begin{align*}
Z_{(1,1,1)} & =\left\{ (z_{0},z_{1},z_{2})\in\mat(2r,3r)\mid\;\det(z_{i},z_{j})\neq0\;(i\ne j)\right\} ,\\
Z_{(2,1)} & =\left\{ (z_{0},z_{1},z_{2})\in\mat(2r,3r)\mid\;\det(z_{0},z_{1})\neq0,\det(z_{0},z_{2})\neq0\right\} ,\\
Z_{(3)} & =\left\{ (z_{0},z_{1},z_{2})\in\mat(2r,3r)\mid\;\det(z_{0},z_{1})\neq0\right\} ,
\end{align*}
where $z_{k}\in\mat(2r,r)$ for $0\leq k\leq2$. To obtain the analogues
of the beta, gamma, Gaussian, we need the normal form of an element
of $Z_{\lm}$ obtained by the action $\GL{2r}\curvearrowright Z_{\lm}\curvearrowleft H_{\lm}$.
The following is Lemma 4.1 of \cite{kimura-2}.
\begin{lem}
\label{lem:ex-1}Let $\lm$ be a partition of $3$. For any $z\in Z_{\lm}$,
we can take a representative $\bx\in Z_{\lm}$ of the orbit $O(z)$
as given in the following table. \bigskip

\begin{tabular}{|c|c|}
\hline 
$\lm$ & normal form $\bx$\tabularnewline
\hline 
\hline 
$(1,1,1)$ & $\left(\begin{array}{ccc}
1_{r} & 0 & 1_{r}\\
0 & 1_{r} & -1_{r}
\end{array}\right)$\tabularnewline
\hline 
$(2,1)$ & $\left(\begin{array}{ccc}
1_{r} & 0 & 0\\
0 & 1_{r} & 1_{r}
\end{array}\right)$\tabularnewline
\hline 
$(3)$ & $\left(\begin{array}{ccc}
1_{r} & 0 & 0\\
0 & 1_{r} & 0
\end{array}\right)$\tabularnewline
\hline 
\end{tabular} \bigskip
\end{lem}

It follows from Lemma \ref{lem:ex-1} that the quotient space $X_{\lm}:=\GL{2r}\backslash Z_{\lm}/H_{\lm}$
consists of one point and realized in $Z_{\lm}$ as 
\begin{align*}
X_{(1,1,1)} & =\left\{ \left(\begin{array}{ccc}
1_{r} & 0 & 1_{r}\\
0 & 1_{r} & -1_{r}
\end{array}\right)\right\} \subset Z_{(1,1,1)},\\
X_{(2,1)} & =\left\{ \left(\begin{array}{ccc}
1_{r} & 0 & 0\\
0 & 1_{r} & 1_{r}
\end{array}\right)\right\} \subset Z_{(2,1)}.\\
X_{(3)} & =\left\{ \left(\begin{array}{ccc}
1_{r} & 0 & 0\\
0 & 1_{r} & 0
\end{array}\right)\right\} \subset Z_{(3)}.
\end{align*}
Using the normal form $\bx\in Z_{\lm}$ given in Lemma \ref{lem:ex-1},
Radon HGF on $X_{\lm}$ is given by

\begin{align*}
F_{(1,1,1)}(\bx,\al;C) & =\int_{C}(\det u)^{\al_{1}}(\det(1_{r}-u))^{\al_{2}}du,\\
F_{(2,1)}(\bx,\al;C) & =\int_{C}\etr\left(\al_{1}u\right)(\det u)^{\al_{2}}du,\\
F_{(3)}(\bx,\al;C) & =\int_{C}\etr\left(\al_{1}u+\al_{2}(-\frac{1}{2}u^{2})\right)du,
\end{align*}
where the integration variable is $u=(u_{i,j})_{1\leq i,j\leq r}$
and $du=\wedge_{i<j}du_{i,j}$. Here the parameters must satisfy the
conditions by definition:
\begin{align*}
\al_{0}+\al_{1}+\al_{2}=-2r & \quad\text{for }\lm=(1,1,1),\\
\al_{0}+\al_{2}=-2r,\;\al_{1}\neq0 & \quad\text{for }\lm=(2,1),\\
\al_{0}=-2r,\;\al_{2}\neq0 & \quad\text{for }\lm=(3).
\end{align*}
Proposition \ref{prop:act-3} tells us that, in the case $\lm=(2,1)$,
we can take $\al_{1}$ to any nonzero complex number, and in the case
$\lm=(3)$, we can take $\al_{1}$ to any complex number and $\al_{2}$
to any nonzero complex number as the effect of action on the Radon
HGF of the continuous part of Weyl group $W_{\lm}$ for $H_{\lm}$.
So we can take the parameter $\al$ so that $\al_{1}=-1$ in case
$\lm=(2,1)$, and $\al_{1}=0,\al_{2}=1$ in case $\lm=(3)$. After
this normalization of the parameter, we have the integrals 
\begin{align*}
F_{(1,1,1)}(\bx,\al;C_{1}) & =\int_{C_{1}}(\det u)^{\al_{1}}(\det(1_{r}-u))^{\al_{2}}du,\\
F_{(2,1)}(\bx,\al;C_{2}) & =\int_{C_{2}}\etr\left(-u\right)(\det u)^{\al_{2}}du,\\
F_{(3)}(\bx,\al;C_{3}) & =\int_{C_{3}}\etr\left(-\frac{1}{2}u^{2}\right)du,
\end{align*}
where the domain of integration $C_{i}$ is a cycle in the homology
group discussed in Section \ref{subsec:Definition-of-HGF}. Note that
the volume form $dU$ in $\herm$ can be written as 
\[
dU=\left(\frac{\sqrt{-1}}{2}\right)^{r(r-1)/2}dU_{1,1}\wedge\cdots\wedge dU_{r,r}\bigwedge_{i\neq j}(dU_{i,j}\wedge dU_{j,i})
\]
and that $\herm$ can be considered as a real form of $\mat(r)$ in
the sense that any $u\in\mat(r)$ can be expressed uniquely as $u=U_{1}+\sqrt{-1}U_{2}$
with $U_{1},U_{2}\in\herm$. So if one restrict $du$ to $\herm$,
$du$ and $dU$ coincide modulo multiplicative constant factor. In
this sense, for the Radon HGF for $\lm=(1,1,1),(2,1),(3)$, the domains
of integration $C_{i}$ are taken in the space $\herm$ as 
\begin{align*}
C_{1} & =\{u\in\herm\mid u>0,1_{r}-u>0\},\\
C_{2} & =\{u\in\herm\mid u>0\},\\
C_{3} & =\herm.
\end{align*}
The integral $F_{(1,1,1)}(\bx,\al;C_{1})$ with $\al=(-a-b,a-r,b-r)$
coincides with $B_{r}(a,b)$ and $F_{(2,1)}(\bx,\al;C_{2})$ with
$\al=(-a-r,-1,a-r)$ coincides with $\G_{r}(a)$ modulo constant factor
$(\sqrt{-1}/2)^{r(r-1)/2}$, respectively. 

Next we explain the effect of finite group part of the Weyl group
stated in Theorem \ref{thm:main-2} when applied to the case $\lm=(1,1,1)$.

\begin{prop}
We have the identity $B_{r}(a,b)=B_{r}(b,a)$.
\end{prop}

\begin{proof}
For $F=F_{(1,1,1)}(\bx,\al;C_{1})$, we apply Theorem \ref{thm:main-2}
(2) with
\[
P_{\sm}=\left(\begin{array}{ccc}
1_{r}\\
 &  & 1_{r}\\
 & 1_{r}
\end{array}\right)\in\cP\simeq\S_{3}
\]
 corresponding the permutation ($\sm(1),\sm(2),\sm(3))=(1,3,2)$ and
obtain
\begin{equation}
F(\bx',\al',C_{1})=F(\bx,\al,C_{1})\label{eq:exa-5}
\end{equation}
 with 
\[
\bx'=\bx P_{\sm}=\left(\begin{array}{ccc}
1_{r} & 1_{r} & 0\\
0 & -1_{r} & 1_{r}
\end{array}\right),\quad\al'=\sm(\al)=(\al_{0},\al_{2},\al_{1}).
\]
We normalize $\bx'$ to $\bx$ by the action of $\GL{2r}$. Put $g=\left(\begin{array}{cc}
1_{r} & 1_{r}\\
0 & -1_{r}
\end{array}\right)$, then we have $g\bx'=\bx$. By applying Proposition \ref{prop:covariance}
to $F(\bx',\al',C_{1})$, we have 
\begin{align}
F(\bx',\al',C_{1}) & =F(g^{-1}\bx,\al',C_{1})=(\det g)^{r}F(\bx,\al',C_{1}')\nonumber \\
 & =(-1)^{r^{2}}F(\bx,\al',C_{1}')=(-1)^{r^{2}}\int_{C_{1}'}(\det u)^{\al_{2}}(\det(1_{r}-u))^{\al_{1}}du,\label{eq:exa-6}
\end{align}
where $C_{1}'$ is the cycle obtained from $C_{1}$ as the image by
the map $\herm\ni u\mapsto1_{r}-u\in\herm$. So $C_{1}$ and $C_{1}'$
is the same as a set but the orientation is different. Noting that
the integral can be reduced to the integral on the variables of eigenvalues
$u_{1},\dots,u_{r}$ of $u$, the cycle $C_{1}$ corresponds to the
cycle in the eigenvalue space given by $(\overrightarrow{0,1})^{r}$
whereas $C_{1}'$ corresponds to $(\overrightarrow{1,0})^{r}$. Hence
we have 
\begin{align}
F(\bx,\al',C_{1}') & =(-1)^{r}\int_{C_{1}}(\det u)^{\al_{2}}(\det(1_{r}-u))^{\al_{1}}du\nonumber \\
 & =(-1)^{r}F(\bx,\al',C_{1}).\label{eq:exa-7}
\end{align}
Thus from (\ref{eq:exa-5}), (\ref{eq:exa-6}) and (\ref{eq:exa-7})
we have $F(\bx',\al',C_{1})=F(\bx,\al,C_{1}),$which implies the identity
$B_{r}(a,b)=B_{r}(b,a)$.
\end{proof}
\begin{rem}
It is desirable to carry out the similar consideration for the other
elements $P_{\sm}$ of $\cP$. But the homology theory associated
with the matrix integral is not sufficiently developed, this is subject
to be considered in the future.
\end{rem}

\subsection{\label{subsec:exa-gauss}Analogues of the Gauss HGF family }

As the classical HGF of one variable, we consider the Gauss HGF and
its confluent family, namely Kummer's confluent HGF, Bessel function,
Hermite-Weber function and Airy function. They are 

\begin{align*}
\text{Gauss: } & \int_{C}u^{a-1}(1-u)^{c-a-1}(1-xu)^{-b}du,\\
\text{Kummer: } & \int_{C}e^{xu}u^{a-1}(1-u)^{c-a-1}du,\\
\text{Bessel: } & \int_{C}e^{u-\frac{x}{u}}u^{-c-1}dt=\int_{C'}e^{xu-\frac{1}{u}}u^{c-1}du,\\
\text{Hermite-Weber: } & \int_{C}e^{xu-\frac{1}{2}u^{2}}u^{-a-1}du,\\
\text{Airy: } & \int_{C}e^{xu-u^{3}/3}du.
\end{align*}
Each of the functions satisfies the 2nd order differential equation
on the complex $x$-plane. For example, the differential equation
for the Gauss HGF and Kummer's confluent HGF are 
\begin{gather}
x(1-x)y''+\{c-(a+b+1)x\}y'-aby=0,\quad'=d/dx,\label{eq:exa-7-1}\\
xy''+(c-x)y'-ay=0,\label{eq:exa-7-2}
\end{gather}
which are called the Gauss hypergeometric equation (Gauss HGE) and
Kummer's confluent hypergeometric equation (Kummer's CHGE), respectively.
Any solution of the differential equations can be represented by the
corresponding integral by choosing an appropriate path of integration
$C$. For example, 
\begin{align*}
\,_{2}F_{1}(a,b,c;x) & =\frac{\G(c)}{\G(a)\G(c-a)}\int_{C}u^{a-1}(1-u)^{c-a-1}(1-xu)^{-b}du,\\
\,_{1}F_{1}(a,c;x) & =\frac{\G(c)}{\G(a)\G(c-a)}\int_{C}e^{xu}u^{a-1}(1-u)^{c-a-1}du
\end{align*}
with the path of integration $C=\overrightarrow{0,1}$, which starts
from $u=0$ and ends at $u=1$, give the holomorphic solutions of
the Gauss HGE and Kummer's CHGE at $x=0$ which take the value $1$
there, respectively.

Hermitian matrix integral analogues of the above family of Gauss HGF
are also considered \cite{Faraut,inamasu-ki,kimura-1,Kontsevich}:
\begin{align}
\text{Gauss: } & \int_{C}|U|^{a-r}|I-U|^{c-a-r}|I-UX|^{-b}\,dU,\nonumber \\
\text{Kummer: } & \int_{C}|U|^{a-r}|I-U|^{c-a-r}\etr(UX)\,dU,\nonumber \\
\text{Bessel: } & \int_{C}|U|^{c-r}\etr(UX-U^{-1})\,dU,\label{eq:exa-7-3}\\
\text{Hermite-Weber: } & \int_{C}|U|^{-c-r}\etr\left(UX-\frac{1}{2}U^{2}\right)\,dU,\nonumber \\
\text{Airy: } & \int_{C}\etr\left(UX-\frac{1}{3}U^{3}\right)\,dU,\nonumber 
\end{align}
where $X\in\herm$, $|U|:=\det U$, $\etr(U):=\exp(\Tr\,U)$ and $dU$
is the volume form on $\herm$ given in (\ref{eq:exa-4}). It is known
that they are functions of eigenvalues $x_{1},\dots,x_{r}$ of $X$
and satisfy holonomic systems of rank $2^{r}$ \cite{kimura-1}. The
Hermitian matrix integral analogue of $\,_{2}F_{1}(a,b,c;x)$ and
$\,_{1}F_{1}(a,c;x)$ are given by 
\begin{align}
\,_{2}\cF_{1}(a,b,c;X) & =\frac{\G_{r}(c)}{\G_{r}(a)\G_{r}(c-a)}\int_{0<U<1_{r}}|U|^{a-r}|I-U|^{c-a-r}|I-UX|^{-b}dU,\label{eq:exa-8}\\
\,_{1}\cF_{1}(a,c;X) & =\frac{\G_{r}(c)}{\G_{r}(a)\G_{r}(c-a)}\int_{0<U<1_{r}}|U|^{a-r}|I-U|^{c-a-r}\etr(UX)\,dU,\label{eq:exa-9}
\end{align}
respectively. To understand these Hermitian matrix integral analogues
as the Radon HGF, we consider the Radon HGFs of type $\lm$, where
$\lm$ is a partition of weight $4$: $(1,1,1,1),(2,1,1),(2,2),(3,1)$
and $(4)$, and the space of independent variables $Z_{\lm}$ is a
subset of $\mat(2r,4r)$. So we consider the following subgroup $\hlam$
of $\GL{4r}$:

\begin{align*}
H_{(1,1,1,1)} & =\left\{ \begin{pmatrix}h_{0}\\
 & h_{1}\\
 &  & h_{2}\\
 &  &  & h_{3}
\end{pmatrix}\right\} , & H_{(2,1,1)} & =\left\{ \begin{pmatrix}h_{0} & h_{1}\\
 & h_{0}\\
 &  & h_{2}\\
 &  &  & h_{3}
\end{pmatrix}\right\} ,\\
H_{(2,2)} & =\left\{ \begin{pmatrix}h_{0} & h_{1}\\
 & h_{0}\\
 &  & h_{2} & h_{3}\\
 &  &  & h_{2}
\end{pmatrix}\right\} , & H_{(3,1)} & =\left\{ \begin{pmatrix}h_{0} & h_{1} & h_{2}\\
 & h_{0} & h_{1}\\
 &  & h_{0}\\
 &  &  & h_{3}
\end{pmatrix}\right\} ,\\
H_{(4)} & =\left\{ \begin{pmatrix}h_{0} & h_{1} & h_{2} & h_{3}\\
 & h_{0} & h_{1} & h_{2}\\
 &  & h_{0} & h_{1}\\
 &  &  & h_{0}
\end{pmatrix}\right\} ,
\end{align*}
 where $h_{k}\in\mat(r)$. A character of the universal covering group
$\tilde{H}_{\lm}$ is given by 
\begin{align*}
\chi_{(1,1,1,1)} & =(\det h_{0})^{\al_{0}}(\det h_{1})^{\al_{1}}(\det h_{2})^{\al_{2}}(\det h_{3})^{\al_{3}},\\
\chi_{(2,1,1)} & =(\det h_{0})^{\al_{0}}\etr(\al_{1}\cdot h_{0}^{-1}h_{1})(\det h_{2})^{\al_{2}}(\det h_{3})^{\al_{3}},\\
\chi_{(2,2)} & =(\det h_{0})^{\al_{0}}\etr(\al_{1}\cdot h_{0}^{-1}h_{1})(\det h_{2})^{\al_{2}}\etr(\al_{3}\cdot h_{2}^{-1}h_{3}),\\
\chi_{(3,1)} & =(\det h_{0})^{\al_{0}}\etr\left(\al_{1}\cdot h_{0}^{-1}h_{1}+\al_{2}\left(h_{0}^{-1}h_{2}-\frac{1}{2}(h_{0}^{-1}h_{1})^{2}\right)\right)(\det h_{3})^{\al_{3}},\\
\chi_{(4)} & =(\det h_{0})^{\al_{0}}\etr\left\{ \al_{1}\cdot h_{0}^{-1}h_{1}+\al_{2}\left(h_{0}^{-1}h_{2}-\frac{1}{2}(h_{0}^{-1}h_{1})^{2}\right)\right.\\
 & \qquad+\left.\al_{3}\left(h_{0}^{-1}h_{3}-(h_{0}^{-1}h_{1})(h_{0}^{-1}h_{2})+\frac{1}{3}(h_{0}^{-1}h_{1})^{3}\right)\right\} .
\end{align*}
By Assumption \ref{assu:Radon-conf-4-1}, the parameter $\al$ satisfies

\begin{align*}
\al_{0}+\al_{1}+\al_{2}+\al_{3}=-2r & \quad\text{for }\lm=(1,1,1,1),\\
\al_{0}+\al_{2}+\al_{3}=-2r,\;\al_{1}\neq0 & \quad\text{for }\lm=(2,1,1),\\
\al_{0}+\al_{2}=-2r,\;\al_{1},\al_{3}\neq0 & \quad\text{for }\lm=(2,2),\\
\al_{0}+\al_{3}=-2r,\;\al_{2}\neq0 & \quad\text{for }\lm=(3,1).\\
\al_{0}=-2r,\;\al_{3}\neq0 & \quad\text{for }\lm=(4).
\end{align*}
The space $Z_{\lm}$ is 
\begin{align*}
Z_{(1,1,1,1)} & =\{(z_{0},z_{1},z_{2},z_{3})\in\mat(2r,4r)\ \mid\ \;\:\det(z_{i},z_{j})\neq0\ (i\neq j)\},\\
Z_{(2,1,1)} & =\left\{ (z_{0},z_{1},z_{2},z_{3})\in\mat(2r,4r)\ \mid\ \begin{aligned} & \det(z_{0},z_{j})\neq0\ (1\leq j\leq3)\\
 & \det(z_{2},z_{3})\neq0
\end{aligned}
\right\} ,\\
Z_{(2,2)} & =\left\{ (z_{0},z_{1},z_{2},z_{3})\in\mat(2r,4r)\ \mid\ \begin{aligned} & \det(z_{0},z_{j})\neq0\ (1\leq j\leq2)\\
 & \det(z_{2},z_{3})\neq0
\end{aligned}
\right\} ,\\
Z_{(3,1)} & =\left\{ (z_{0},z_{1},z_{2},z_{3})\in\mat(2r,4r)\ \mid\ \begin{aligned} & \det(z_{0},z_{1})\neq0\\
 & \det(z_{0},z_{3})\neq0
\end{aligned}
\right\} ,\\
Z_{(4)} & =\left\{ (z_{0},z_{1},z_{2},z_{3})\in\mat(2r,4r)\ \mid\ \;\:\det(z_{0},z_{1})\neq0\right\} ,
\end{align*}
where $z_{k}\in\mat(2r,r)$. It is seen that $Z_{\lm}$ is invariant
by the action of $\GL{2r}\times H_{\lm}$. We will see that these
partitions correspond to the Radon HGF analogues of the Gauss, Kummer,
Bessel, Hermite-Weber and Airy functions, respectively. To see this,
we need to normalize $z\in Z_{\lm}$ by the action of $\GL{2r}\times H_{\lm}$
and to normalize the parameter $\al=(\al_{0},\al_{1},\al_{2},\al_{3})$
for $\chi_{\lm}(\cdot,\al)$ using the continuous part $U$ of the
Weyl group $W_{\lm}$.
\begin{prop}
\label{prop:ex-2}For partitions $\lm$ of $4$, we have the following.

(1) Any element $z\in Z_{\lm}$ is taken to the form $\bx$ by the
action of $\GL{2r}\times H_{\lm}$ as given in the table below. 

(2) The parameter $\al=(\al_{0},\al_{1},\al_{2},\al_{3})$ is taken
to the normal form by the action of continuous part of the Weyl group
$W_{\lm}$ as given in the table below.\bigskip

\begin{tabular}{|c|c|c|c|}
\hline 
$\lm$ & Normal form $\bx$ & Normal form of $\al$ & Condition for $\al$\tabularnewline
\hline 
\hline 
$1,1,1,1$ & $\left(\begin{array}{cccc}
1_{r} & 0 & 1_{r} & 1_{r}\\
0 & 1_{r} & -1_{r} & -x
\end{array}\right)$ & $\al_{0},\al_{1},\al_{2},\al_{3}$ & $\al_{0}+\al_{1}+\al_{2}+\al_{3}=-2r$\tabularnewline
\hline 
$2,1,1$ & $\left(\begin{array}{cccc}
1_{r} & 0 & 0 & 1_{r}\\
0 & x & 1_{r} & -1_{r}
\end{array}\right)$ & $\al_{0},1,\al_{2},\al_{3}$ & $\al_{0}+\al_{2}+\al_{3}=-2r$\tabularnewline
\hline 
$2,2$ & $\left(\begin{array}{cccc}
1_{r} & 0 & 0 & -1_{r}\\
0 & x & 1_{r} & 0
\end{array}\right)$ & $\al_{0},1,\al_{2},1$ & $\al_{0}+\al_{2}=-2r$\tabularnewline
\hline 
$3,1$ & $\left(\begin{array}{cccc}
1_{r} & 0 & 0 & 0\\
0 & 1_{r} & x & 1_{r}
\end{array}\right)$ & $\al_{0},0,1,\al_{3}$ & $\al_{0}+\al_{3}=-2r$\tabularnewline
\hline 
$4$ & $\left(\begin{array}{cccc}
1_{r} & 0 & 0 & 0\\
0 & 1_{r} & 0 & -x
\end{array}\right)$ & $-2r,0,0,1$ & \tabularnewline
\hline 
\end{tabular}
\end{prop}

\begin{proof}
The assertion (1) was obtained as Lemma 4.3 of \cite{kimura-2}. The
assertion (2) is a consequence of Proposition \ref{prop:act-3}.
\end{proof}
Considering the Radon HGF for the normalized variable $\bx$ and the
parameter $\al$ given in the previous proposition, we have the Radon
HGF which give the Hermitian matrix integral analogue (\ref{eq:exa-7-3})
of the Gauss HGF family as follows.
\begin{prop}
\label{prop:ex-3}Let  $\lm,\bx$  and  $\al$  be as in Proposition
\ref{prop:ex-2}. Then the corresponding Radon HGF $F_{\lm}(z;\al)$
is expressed as  
\begin{align*}
F_{(1,1,1,1)}(\bx;\al) & =\int_{C}(\det u)^{\al_{1}}(\det(1_{r}-u))^{\al_{2}}(\det(1_{r}-ux_{1}))^{\al_{3}}du,\\
F_{(2,1,1)}(\bx;\al) & =\int_{C}\etr(ux)(\det u)^{\al_{2}}(\det(1_{r}-u))^{\al_{3}}du,\\
F_{(2,2)}(\bx;\al) & =\int_{C}\etr(ux-u^{-1})(\det u)^{\al_{2}}du,\\
F_{(3,1)}(\bx;\al) & =\int_{C}\etr\left(ux-\frac{1}{2}u^{2}\right)(\det u)^{\al_{3}}du\\
F_{(4)}(\bx;\al) & =\int_{C}\etr\left(ux-\frac{1}{3}u^{3}\right)du,
\end{align*}
 where $u\in\mat(r)$ and $C$ denotes a domain of integration.
\end{prop}

Let us see what can be obtained as the effect of the finite group
part of Weyl group $W_{\lm}$ when Theorem \ref{thm:main-2} is applied
to the above Radon HGF analogues. 

\subsubsection{Transformation formula for $\protect\lm=(1,1,1,1)$}

To understand the meaning the formula (\ref{eq:action-8}) for the
Radon HGF corresponding to the Gauss HGF, let us recall the formulae
for the classical case called ``Kummer's 24 solutions'' for the
Gauss HGE (\ref{eq:exa-7-1}). We know 
\begin{align*}
\ghyp(a,b,c;x) & =(1-x)^{-a}\,\ghyp\left(c-b,a,c\,;\frac{x}{x-1}\right)\\
 & =(1-x)^{-b}\,\ghyp\left(c-a,b,c\,;\frac{x}{x-1}\right)\\
 & =(1-x)^{c-a-b}\,\ghyp(c-a,c-b,c\,;x),
\end{align*}
 and 
\begin{align*}
 & x^{1-c}\,\ghyp(a+1-c,b+1-c,2-c;x)\\
 & \qquad=x^{1-c}(1-x)^{c-a-1}\,\ghyp\left(1-b,a+1-c,2-c\,;\frac{x}{x-1}\right)\\
 & \qquad=x^{1-c}(1-x)^{c-b-1}\,\ghyp\left(1-a,b+1-c,2-c\,;\frac{x}{x-1}\right)\\
 & \qquad=x^{1-c}(1-x)^{c-a-b}\,\ghyp(1-a,1-b,2-c\,;x).
\end{align*}
The first four expressions represent the solution of the Gauss HGE
at $x=0$ having the characteristic exponent $0$. On the other hand
the second four expressions represent the solution at $x=0$ with
the exponent $1-c$. In total we have $8$ expressions for solutions
of Gauss HGE at $x=0$. We also have $8$ expressions of solutions
at each singular point $x=1,\infty$, and hence 24 expressions for
the solutions of the Gauss HGE in total. 

We shall derive a similar formula for the Radon HGF $F_{(1,1,1,1)}(\bx;\al)$,
which will be simply denoted as $F(\bx,\al)$. We also use the notation
$H:=H_{(1,1,1,1)}$. Put 
\begin{equation}
X=\left\{ \bx=\left(\begin{array}{cccc}
1_{r} & 0 & 1_{r} & 1_{r}\\
0 & 1_{r} & -1_{r} & -x
\end{array}\right)\in Z_{\lm}\right\} \label{eq:exa-10}
\end{equation}
and let us compute the transformations of $X$ which are induced from
the action of the discrete part $\cP$ of the Weyl group $W:=W_{(1,1,1,1)}$.
Note that $\cP\simeq\S_{4}$ and $\sm\in\S_{4}$ is identified with
the permutation matrix $P_{\sm}=(\de_{\sm(i),j}\cdot1_{r})_{0\leq i,j\leq3}\in\GL{4r}$.
Let $K$ be the subgroup of $\S_{4}$ given by 
\[
K=\{\mathrm{id},(0,1)(2,3),(0,2)(1,3),(0,3)(1,2)\}.
\]
Firstly we consider the action of transpositions $\sm$.
\begin{lem}
\label{lem:ex-gauss-1}The action of $P_{\sm}$ for transpositions
$\sm$ on $X$ and on the space of $\al$ is described in the following
table: \medskip

\begin{tabular}{|c|c|c|}
\hline 
$\sm$ & Transformation of $\bx$ & Transformation of $\al$\tabularnewline
\hline 
\hline 
$(0,1)$ & $x\to x^{-1}$ & $\al\to(\al_{1},\al_{0},\al_{2},\al_{3})$\tabularnewline
\hline 
$(0,2)$ & $x\to1_{r}-x$ & $\al\to(\al_{2},\al_{1},\al_{0},\al_{3})$\tabularnewline
\hline 
$(0,3)$ & $x\to x(x-1_{r})^{-1}$ & $\al\to(\al_{3},\al_{1},\al_{2},\al_{0})$\tabularnewline
\hline 
$(1,2)$ & $x\to x(x-1_{r})^{-1}$ & $\al\to(\al_{0},\al_{2},\al_{1},\al_{3})$\tabularnewline
\hline 
$(1,3)$ & $x\to1_{r}-x$ & $\al\to(\al_{0},\al_{3},\al_{2},\al_{1})$\tabularnewline
\hline 
$(2,3)$ & $x\to x^{-1}$ & $\al\to(\al_{0},\al_{1},\al_{3},\al_{2})$\tabularnewline
\hline 
\end{tabular}\bigskip 
\end{lem}

\begin{proof}
1) \emph{Case $\sm=(0,1)$}. Apply $P_{\sm}$ to $\bx$ in (\ref{eq:exa-10}):
\[
\bx=\left(\begin{array}{cccc}
1_{r} & 0 & 1_{r} & 1_{r}\\
0 & 1_{r} & -1_{r} & -x
\end{array}\right)\mapsto z:=\bx P_{\sm}=\left(\begin{array}{cccc}
0 & 1_{r} & 1_{r} & 1_{r}\\
1_{r} & 0 & -1_{r} & -x
\end{array}\right).
\]
We transform $z$ to an element of $X$ by the action of $\GL{2r}\times H$.
Take
\[
g=\left(\begin{array}{cc}
0 & 1_{r}\\
1_{r} & 0
\end{array}\right)\in\GL{2r},\quad h=\diag(1_{r},1_{r},h_{2},h_{3})\in H
\]
and, noting $g=g^{-1}$, get 
\[
g^{-1}zh=\left(\begin{array}{cccc}
1_{r} & 0 & -h_{2} & -xh_{3}\\
0 & 1_{r} & h_{2} & h_{3}
\end{array}\right).
\]
 So we take $h$ as $h=\diag(1_{r},1_{r},-1_{r},-x^{-1})$ and get
\[
\bx'=g^{-1}zh=\left(\begin{array}{cccc}
1_{r} & 0 & 1_{r} & 1_{r}\\
0 & 1_{r} & -1_{r} & -x^{-1}
\end{array}\right).
\]
Thus we obtain the transformation $X\ni\bx\mapsto\bx'\in X$ which
is equivalent to $\GL r\ni x\mapsto x^{-1}\in\GL r$. 

2) \emph{Case $\sm=(0,2)$}. Consider $\bx\mapsto z:=\bx P_{\sm}$.
Then 
\[
z=\left(\begin{array}{cccc}
1_{r} & 0 & 1_{r} & 1_{r}\\
-1_{r} & 1_{r} & 0 & -x
\end{array}\right).
\]
We transform $z$ to an element of $X$ by the action of $\GL{2r}\times H$.
Take
\[
g_{1}=\left(\begin{array}{cc}
1_{r} & 0\\
-1_{r} & 1_{r}
\end{array}\right)\in\GL{2r},\quad h=\diag(h_{0},1_{r},h_{2},h_{3})
\]
and get 
\[
g_{1}^{-1}zh=\left(\begin{array}{cccc}
h_{0} & 0 & h_{2} & h_{3}\\
0 & 1_{r} & h_{2} & (1_{r}-x)h_{3}
\end{array}\right)
\]
Further we take $g_{2}=\diag(h_{0},1_{r})$ to obtain
\[
\bx'=g_{2}^{-1}g_{1}^{-1}zh=\left(\begin{array}{cccc}
1_{r} & 0 & h_{0}^{-1}h_{2} & h_{0}^{-1}h_{3}\\
0 & 1_{r} & h_{2} & (1_{r}-x)h_{3}
\end{array}\right).
\]
So we determine $h_{0}=h_{2}=h_{3}=-1_{r}$ to obtain 
\[
\bx'=\left(\begin{array}{cccc}
1_{r} & 0 & 1_{r} & 1_{r}\\
0 & 1_{r} & -1_{r} & -(1_{r}-x)
\end{array}\right).
\]
Thus we obtain the transformation $x\mapsto1_{r}-x$ induced by the
action of $P_{\sm}$.

3) \emph{Case $\sm=(0,3)$}. Consider the change $\bx\mapsto z:=\bx P_{\sm}$.
We take $z$ to the normal form.
\[
z=\left(\begin{array}{cccc}
1_{r} & 0 & 1_{r} & 1_{r}\\
-x & 1_{r} & -1_{r} & 0
\end{array}\right)\to g_{1}^{-1}zh=\left(\begin{array}{cccc}
h_{0} & 0 & h_{2} & h_{3}\\
0 & 1_{r} & -(1_{r}-x)h_{2} & xh_{3}
\end{array}\right)
\]
by taking $g_{1}\in\GL{2r}$ and $h\in H$ as 
\[
g_{1}=\left(\begin{array}{cc}
1_{r} & 0\\
-x & 1_{r}
\end{array}\right),\quad h=\diag(h_{0},1_{r},h_{2},h_{3}).
\]
Multiply $g_{2}^{-1}=\diag(h_{0}^{-1},1_{r})$ to $g_{1}^{-1}zh$
from the left to obtain
\[
\bx'=g_{2}^{-1}g_{1}^{-1}zh=\left(\begin{array}{cccc}
1_{r} & 0 & h_{0}^{-1}h_{2} & h_{0}^{-1}h_{3}\\
0 & 1_{r} & -(1_{r}-x)h_{2} & xh_{3}
\end{array}\right).
\]
So we take $h$ so that $h_{0}^{-1}h_{2}=h_{0}^{-1}h_{3}=(1_{r}-x)h_{2}=1_{r}$,
namely $h_{0}=h_{2}=h_{3}=(1_{r}-x)^{-1}$, to obtain 
\[
\bx'=\left(\begin{array}{cccc}
1_{r} & 0 & 1_{r} & 1-r\\
0 & 1_{r} & -1_{r} & -x'
\end{array}\right),\quad x'=x(x-1_{r})^{-1}.
\]
Thus we obtain the transformation $x\mapsto x(x-1_{r})^{-1}$. For
$\sm=(1,2),(1,3),(2,3)$, it is sufficient to note that
\[
(1,2)=(0,1)(0,2)(0,1),\quad(1,3)=(0,1)(0,3)(0,1),\quad(2,3)=(0,2)(0,3)(0,2).
\]
For example, for $\sm=(1,2)$, the transformation can be computed
as 
\[
x\overset{P_{(0,1)}}{\longrightarrow}x^{-1}\overset{P_{(0,2)}}{\longrightarrow}1_{r}-x^{-1}\overset{P_{(0,1)}}{\longrightarrow}x(x-1_{r})^{-1}.
\]
The cases $\sm=(1,3),(2,3)$ can be treated in a similar way.
\end{proof}
From Lemma \ref{lem:ex-gauss-1}, we have the following result.
\begin{cor}
The transformation of $X$ induced by $P_{\sm}\in\cP\;(\sm\in K)$
is the identity and $P_{\sm}$ induces the transformation of parameter
$\al=(\al_{0},\al_{1},\al_{2},\al_{3})$ given as \medskip 

\begin{tabular}{|c|c|}
\hline 
$\sm$ & Transformation of $\al$\tabularnewline
\hline 
\hline 
$\mathrm{id}$ & $\al\to\al$\tabularnewline
\hline 
$(0,1)(2,3)$ & $\al\to(\al_{1},\al_{0},\al_{3},\al_{2})$\tabularnewline
\hline 
$(0,2)(1,3)$ & $\al\to(\al_{2},\al_{3},\al_{0},\al_{3})$\tabularnewline
\hline 
$(0,3)(1,2)$ & $\al\to(\al_{3},\al_{2},\al_{1},\al_{0})$\tabularnewline
\hline 
\end{tabular}.\bigskip 
\end{cor}

Note that $K$ is a normal subgroup of $\S_{4}$ and the nontrivial
transformations of $X$ comes from the representatives of the quotient
group $\S_{4}\slash K\simeq\S_{3}$. 
\begin{cor}
The representatives of the group $\S_{4}\slash K$, the transformation
of $X$ and of the parameter $\al$ are given by the following table.
\end{cor}

\medskip

\begin{tabular}{|c|c|c|}
\hline 
Representative $\sm$ & Transformation of $\bx$ & Transformation of $\al$\tabularnewline
\hline 
\hline 
$\mathrm{id}$ & $x\to x$ & $\al\to\al$\tabularnewline
\hline 
$(0,1)$ & $x\to x^{-1}$ & $\al\to(\al_{1},\al_{0},\al_{2},\al_{3})$\tabularnewline
\hline 
$(0,2)$ & $x\to1_{r}-x$ & $\al\to(\al_{2},\al_{1},\al_{0},\al_{3})$\tabularnewline
\hline 
$(0,3)$ & $x\to x(x-1_{r})^{-1}$ & $\al\to(\al_{3},\al_{1},\al_{2},\al_{0})$\tabularnewline
\hline 
$(0,1,2)$ & $x\to(1_{r}-x)^{-1}$ & $\al\to(\al_{1},\al_{2},\al_{0},\al_{3})$\tabularnewline
\hline 
$(0,2,1)$ & $x\to(x-1_{r})x^{-1}$ & $\al\to(\al_{2},\al_{0},\al_{1},\al_{3})$\tabularnewline
\hline 
\end{tabular}.\bigskip 

Let us derive some transformation formula for (\ref{eq:exa-8}) applying
Theorem \ref{thm:main-2} (2) for the elements of $\S_{4}$ which
preserve the indices $\{1,2\}$ as a set:
\begin{equation}
Q=\{\mathrm{id},(1,2),(0,3),(0,3)(1,2)\}\simeq\Z_{2}\times\Z_{2}.\label{eq:exa-11}
\end{equation}
Recall that the Hermitian matrix integral analogue (\ref{eq:exa-8})
of the Gauss HGF is related to the Radon HGF of type $\lm=(1,1,1,1)$
as 
\begin{align*}
\,_{2}\cF_{1}(a,b,c;x) & =A(\al)F(\bx,\al;C)\\
 & =A(\al)\int_{C}(\det u)^{\al_{1}}\left(\det(1_{r}-u)\right)^{\al_{2}}\left(\det(1_{r}-ux)\right)^{\al_{3}}du,
\end{align*}
where
\[
\bx=\left(\begin{array}{cccc}
1_{r} & 0 & 1_{r} & 1_{r}\\
0 & 1_{r} & -1_{r} & -x
\end{array}\right),\quad C=\{u\in\herm\mid u>0,1_{r}-u>0\}
\]
 and 
\[
(a,b,c)=(\al_{1}+r,-\al_{3},\al_{1}+\al_{2}+2r),\quad A(\al)=\left(\frac{\sqrt{-1}}{2}\right)^{r(r-1)/2}\frac{\G_{r}(\al_{1}+\al_{2}+2r)}{\G_{r}(\al_{1}+r)\G_{r}(\al_{2}+r)}.
\]

\begin{prop}
\label{prop:ex-gauss-2}Let $Q$ be the group defined by (\ref{eq:exa-11}).
Then the action of $(1,2)\in Q$ on the Radon HGF induces

\[
\,_{2}\cF_{1}(a,b,c;x)=\left(\det(1_{r}-x)\right)^{-b}\,_{2}\cF_{1}(c-a,b,c;x(x-1_{r})^{-1}).
\]
\end{prop}

\begin{proof}
For $(1,2)\in Q$, we compute an identity for $\,_{2}\cF_{1}(a,b,c;X)$
derived by virtue of Theorem \ref{thm:main-2}. Put $z:=\bx P_{\sm}$
and normalize it to the normal form by the action of $\GL{2r}\times H$:
\begin{align*}
z & =\left(\begin{array}{cccc}
1_{r} & 1_{r} & 0 & 1_{r}\\
0 & -1_{r} & 1_{r} & -x
\end{array}\right)\mapsto g^{-1}z=\left(\begin{array}{cccc}
1_{r} & 0 & 1_{r} & 1_{r}-x\\
0 & 1_{r} & -1_{r} & x
\end{array}\right)\\
 & \mapsto g^{-1}zh=\left(\begin{array}{cccc}
1_{r} & 0 & 1_{r} & (1_{r}-x)h_{3}\\
0 & 1_{r} & -1_{r} & xh_{3}
\end{array}\right),
\end{align*}
where 
\[
g=\left(\begin{array}{cc}
1_{r} & 1_{r}\\
0 & -1_{r}
\end{array}\right)\in\GL{2r},\quad h=\diag(1_{r},1_{r},1_{r},h_{3})\in H.
\]
 So we take $h_{3}=(1_{r}-x)^{-1}$ and obtain 
\[
\bx':=g^{-1}zh=\left(\begin{array}{cccc}
1_{r} & 0 & 1_{r} & 1_{r}\\
0 & 1_{r} & -1_{r} & -x'
\end{array}\right),\quad x'=x(x-1_{r})^{-1}.
\]
 Note that $A(\al)=A(\sm(\al))$, since $\sm$ exchange $\al_{1}$
and $\al_{2}$ and $A(\al)$ is invariant by this exchange. Then
\begin{align*}
\,_{2}\cF_{1}(a,b,c;x) & =A(\al)F(\bx,\al;C)\\
 & =A(\sm(\al))F(\bx P_{\sm},\sm(\al);C)\\
 & =A(\sm(\al))F(g\bx'h^{-1},\sm(\al);C)\\
 & =A(\sm(\al))(\det g)^{-r}\chi(h,\sm(\al))^{-1}F(\bx',\sm(\al);C').
\end{align*}
where $C'$ is obtained from $C$ as the image of the map $\mat(r)\ni u\mapsto1_{r}-u\in\mat(r)$.
Since $C\subset\herm$, this map is considered from $\herm$ to itself.
Hence $C'=(-1)^{r}C$. Noting $\det g=(-1)^{r}$, we have 
\begin{align*}
\,_{2}\cF_{1}(a,b,c;x) & =A(\sm(\al))\left(\det(1_{r}-x)\right)^{\al_{3}}F(\bx',\sm(\al);C)\\
 & =\left(\det(1_{r}-x)\right)^{\al_{3}}\,_{2}\cF_{1}(\al_{2}+r,-\al_{3},\al_{1}+\al_{2}+2r;x(x-1_{r})^{-1})\\
 & =\left(\det(1_{r}-x)\right)^{-b}\,_{2}\cF_{1}(c-a,b,c;x(x-1_{r})^{-1}).
\end{align*}
\end{proof}
\begin{rem}
Similar result as in Proposition \ref{prop:ex-gauss-2} is given as
Proposition XV.3.4. (i) in \cite{Faraut}.
\end{rem}

\begin{conjecture}
The action of $(0,3),(1,2)(0,3)\in Q$ on the Radon HGF induces the
identities:

\begin{align}
\,_{2}\cF_{1}(a,b,c;X) & =\left(\det(1_{r}-X)\right)^{-a}\,_{2}\cF_{1}(a,c-b,c;X(X-1_{r})^{-1})\label{eq:conj-1}\\
 & =\left(\det(1_{r}-X)\right)^{c-a-b}\,_{2}\cF_{1}(c-a,c-b,c;X).\label{eq:conj-2}
\end{align}
\end{conjecture}

We explain why this statement stays at the level of a conjecture.
When we try to prove the formulas (\ref{eq:conj-1}) and (\ref{eq:conj-2})
applying similar argument as in the proof of Proposition \ref{prop:ex-gauss-2}
for $\sm=(0,3)$ and $\sm=(1,2)(0,3)$, we encounter the problem of
homology group as follows.

\emph{1) Case $\sm=(0,3)$}. The normalization of $z:=\bx P_{\sm}$
by the action of $\GL{2r}\times H$ is already treated in the proof
of Lemma \ref{lem:ex-gauss-1}. We have 
\[
\bx'=g^{-1}zh=\left(\begin{array}{cccc}
1_{r} & 0 & 1_{r} & 1_{r}\\
0 & 1_{r} & -1_{r} & -x'
\end{array}\right),\quad x'=x(x-1_{r})^{-1}
\]
by 
\[
g=\left(\begin{array}{cc}
(1_{r}-x)^{-1} & 0\\
-x(1_{r}-x)^{-1} & 1_{r}
\end{array}\right),\quad h=\diag((1_{r}-x)^{-1},1_{r},(1_{r}-x)^{-1},(1_{r}-x)^{-1})
\]
Note that $A(\al)=A(\sm(\al))$, since $\al_{1}$ and $\al_{2}$ are
fixed by $\sm$. Then we have 
\begin{align*}
\,_{2}\cF_{1}(a,b,c;x) & =A(\al)F(\bx,\al;C)\\
 & =A(\sm(\al))F(\bx P_{\sm},\sm(\al);C)\\
 & =A(\sm(\al))F(g\bx'h^{-1},\sm(\al);C)\\
 & =A(\sm(\al))(\det g)^{-r}\chi(h,\sm(\al))^{-1}F(\bx',\sm(\al);C').
\end{align*}
where $C'$ is obtained from $C$ as the image of the map
\begin{align*}
u & \mapsto(1_{r},u)g=((1_{r}-ux)(1_{r}-x)^{-1},u)\\
 & \mapsto(1_{r},(1_{r}-x)(1_{r}-ux)^{-1}u)\mapsto(1_{r}-x)(1-ux)^{-1}u.
\end{align*}
 \emph{We may have} $C'\sim C$, \emph{namely $C'$ and $C$ are homologous
in this case since $C'$ can be deformed continuously to $C$ as $x\to0$.}
(This statement is not obvious because of the lack of knowledge on
the homology group). Since $(\det g)^{-r}=\left(\det(1_{r}-x)\right)^{r}$,
we have 
\begin{align*}
\,_{2}\cF_{1}(a,b,c;x) & =A(\sm(\al))\left(\det(1_{r}-x)\right)^{\al_{0}+\al_{2}+\al_{3}+r}F(\bx',\sm(\al);C)\\
 & =\left(\det(1_{r}-x)\right)^{-\al_{1}-r}\,_{2}\cF_{1}(\al_{1}+r,-\al_{0},\al_{1}+\al{}_{2}+2r;x(x-1_{r})^{-1})\\
 & =\left(\det(1_{r}-x)\right)^{-a}\,_{2}\cF_{1}(a,c-b,c;x(x-1_{r})^{-1}).
\end{align*}

\emph{2) Case $\sm=(1,2)(0,3)$}. Put $z:=\bx P_{\sm}$ and normalize
it to the normal form by the action of $\GL{2r}\times H$:
\begin{align*}
z & =\left(\begin{array}{cccc}
1_{r} & 1_{r} & 0 & 1_{r}\\
-x & -1_{r} & 1_{r} & 0
\end{array}\right)\mapsto g^{-1}z=\left(\begin{array}{cccc}
1_{r} & 0 & -(x-1_{r})^{-1} & -(x-1_{r})^{-1}\\
0 & 1_{r} & (x-1_{r})^{-1} & (x-1_{r})^{-1}x
\end{array}\right)\\
 & \mapsto g^{-1}zh=\left(\begin{array}{cccc}
1_{r} & 0 & -(x-1_{r})^{-1}h_{2} & -(x-1_{r})^{-1}h_{3}\\
0 & 1_{r} & (x-1_{r})^{-1}h_{2} & (x-1_{r})^{-1}xh_{3}
\end{array}\right),
\end{align*}
where 
\[
g=\left(\begin{array}{cc}
1_{r} & 1_{r}\\
-x & -1_{r}
\end{array}\right)\in\GL{2r},\quad h=\diag(1_{r},1_{r},h_{2},h_{3})\in H.
\]
So noting $x$ commutes with $(x-1_{r})^{-1}$, we choose $h$ so
that $h_{2}=h_{3}=1_{r}-x$ to obtain 
\[
\bx'=\left(\begin{array}{cccc}
1_{r} & 0 & 1_{r} & 1_{r}\\
0 & 1_{r} & -1_{r} & -x'
\end{array}\right),\quad x'=x.
\]
 Note that $A(\al)=A(\sm(\al))$.
\begin{align*}
\,_{2}\cF_{1}(a,b,c;x) & =A(\al)F(\bx,\al;C)\\
 & =A(\sm(\al))F(\bx P_{\sm},\sm(\al);C)\\
 & =A(\sm(\al))F(g^{-1}\bx'h^{-1},\sm(\al);C)\\
 & =A(\sm(\al))(\det g)^{r}\chi(h,\sm(\al))^{-1}F(\bx',\sm(\al);C'),
\end{align*}
where $C'$ is obtained from $C$ as the image of the map
\begin{align*}
u & \mapsto(1_{r},u)g=(-(1-ux)(x-1_{r})^{-1},-(1_{r}-u)(x-1_{r})^{-1}).\\
 & \mapsto(1_{r},(x-1_{r})(1_{r}-ux)^{-1}(1_{r}-u)(x-1_{r})^{-1})\\
 & \mapsto(x-1_{r})(1_{r}-ux)^{-1}(1_{r}-u)(x-1_{r})^{-1}
\end{align*}
\emph{We see that $C'$ is homologous to $(-1)^{r}C$ since the above
transformation is deformed continuously to $u\mapsto1-u$ when $x$
moves continuously to $0$ }(This statement is not obvious)\emph{.}
Since $(\det g)^{r}=\left(\det(x-1_{r})\right)^{-r}$, we have
\begin{align*}
\,_{2}\cF_{1}(a,b,c;x) & =A(\sm(\al))(-1)^{r}\left(\det(x-1_{r})\right)^{-r}\left(\det(1_{r}-x)\right)^{-\al_{0}-\al_{1}}F(\bx',\sm(\al);C)\\
 & =\left(\det(1_{r}-x)\right)^{-\al_{0}-\al_{1}-r}\,_{2}\cF_{1}(\al_{2}+r,-\al_{0},\al_{1}+\al_{2}+2r;x)\\
 & =\left(\det(1_{r}-x)\right)^{c-a-b}\,_{2}\cF_{1}(c-a,c-b,c;x).
\end{align*}

\subsubsection{Transformation formula for Kummer's analogue}

The following is known as the Kummer's first transformation formula
for the classical Kummer's confluent HGF:
\[
_{1}F_{1}(a,c;x)=e^{x}\cdot{}_{1}F_{1}(c-a,c;-x).
\]
 We shall derive an analogous formula for its Hermitian matrix integral
analogue $\,_{1}\cF_{1}(a,c;x)$ defined by (\ref{eq:exa-9}) applying
Theorem \ref{thm:main-2} to the corresponding Radon HGF for the partition
$\lm=(2,1,1)$: 
\[
F(\bx,\al;C):=F_{(2,1,1)}(\bx,\al;C)=\int_{C}\etr(ux)(\det u)^{\al_{2}}(\det(1_{r}-u))^{\al_{3}}du,
\]
where
\[
\bx=\left(\begin{array}{cccc}
1_{r} & 0 & 0 & 1_{r}\\
0 & x & 1_{r} & -1_{r}
\end{array}\right),\quad\al=(\al_{0},1,\al_{2},\al_{3}).
\]
Recall that $\,_{1}\cF_{1}(a,c;x)$ and $F(\bx,\al;C)$ are related
as

\begin{equation}
\,_{1}\cF_{1}(a,c;x)=A(\al)F(\bx,\al;C)\label{eq:ex-7}
\end{equation}
 with $C=\{u\in\herm\mid u>0,1_{r}-u>0\}$ and 
\[
(a,c)=(\al_{2}+r,\al_{2}+\al_{3}+2r),\quad A(\al)=\frac{\G_{r}(\al_{2}+\al_{3}+2r)}{\G_{r}(\al_{2}+r)\G_{r}(\al_{3}+r)}.
\]
The finite group part $\cP$ of the Weyl group $W_{\lm}$ is isomorphic
to the permutation group $\S_{2}$ and is generated by the permutation
matrix 
\begin{equation}
P_{\sm}=\left(\begin{array}{cccc}
1_{r}\\
 & 1_{r}\\
 &  &  & 1_{r}\\
 &  & 1_{r}
\end{array}\right)\label{eq:ex-8}
\end{equation}
associated with the transposition $\sm=(2,3)$ which acts on $Z_{\lm}$
and on the weights $\al$ by $z\mapsto zP_{\sm}$ and $\al\mapsto\sm(\al)$,
respectively.
\begin{prop}
\label{prop:ex-kummer-1}For the Kummer's analogue (\ref{eq:exa-9}),
we have
\[
\,_{1}\cF_{1}(a,c;x)=\etr(x)\cdot\,_{1}\cF_{1}(c-a,c;-x).
\]
\end{prop}

\begin{proof}
By Theorem \ref{thm:main-2}, we have
\[
F(\bx,\al;C)=F(\bx P_{\sm},\sm(\al);C)
\]
for $P_{\sm}$ given by (\ref{eq:ex-8}) and $\sm=(2,3)$. Put $z:=\bx P_{\sm}$
and normalize it by the action $\GL{2r}\times H_{(2,1,1)}$:
\begin{align*}
z & =\left(\begin{array}{cccc}
1_{r} & 0 & 1_{r} & 0\\
0 & x & -1_{r} & 1_{r}
\end{array}\right)\mapsto g^{-1}z=\left(\begin{array}{cccc}
1_{r} & x & 0 & 1_{r}\\
0 & -x & 1_{r} & -1_{r}
\end{array}\right)\\
 & \mapsto g^{-1}zh=\left(\begin{array}{cccc}
1_{r} & h_{1}+x & 0 & 1_{r}\\
0 & -x & 1_{r} & -1_{r}
\end{array}\right),
\end{align*}
where 
\begin{equation}
g=g^{-1}=\left(\begin{array}{cc}
1_{r} & 1_{r}\\
0 & -1_{r}
\end{array}\right)\in\GL{2r},\quad h=\left(\begin{array}{cc}
1_{r} & h_{1}\\
 & 1_{r}
\end{array}\right)\oplus\diag(1_{r},1_{r})\in H_{(2,1,1)}.\label{eq:ex-9}
\end{equation}
 So we take $h_{1}=-x$ and get
\[
\bx':=g^{-1}zh=\left(\begin{array}{cccc}
1_{r} & 0 & 0 & 1_{r}\\
0 & x' & 1_{r} & -1_{r}
\end{array}\right),\quad x'=-x.
\]
 Noting $A(\al)=A(\sm(\al))$, we see
\begin{align*}
\,_{1}\cF_{1}(a,c;x) & =A(\al)F(\bx,\al;C)\\
 & =A(\sm(\al))F(\bx P_{\sm},\sm(\al);C)\\
 & =A(\sm(\al))F(g\bx'h^{-1},\sm(\al);C)\\
 & =A(\sm(\al))(\det g)^{-r}\chi(h,\sm(\al))^{-1}F(\bx',\sm(\al);C').
\end{align*}
where $C'$ is obtained from $C$ as the image of the map $\mat(r)\ni u\mapsto1_{r}-u\in\mat(r).$
Since $C\subset\herm$, this map is considered from $\herm$ to itself.
Hence $C'=(-1)^{r}C$. Noting $\det g=(-1)^{r}$, we have 
\begin{align*}
\,_{1}\cF_{1}(a,c;x) & =A(\sm(\al))\etr(x)\cdot F(\bx',\sm(\al);C)\\
 & =\etr(x)\cdot\,_{1}\cF_{1}(\al_{3}+r,\al_{2}+\al_{3}+2r;-x)\\
 & =\etr(x)\cdot\,_{1}\cF_{1}(c-a,c;-x).
\end{align*}
\end{proof}

\end{document}